\documentclass[11pt]{amsart}

\makeindex

\usepackage{geometry}                
\usepackage{amssymb}
\usepackage{palatino}
\usepackage{mathpazo}

\usepackage[dvips]{graphics}
\usepackage[dvips]{graphicx}

\usepackage{amsmath}
\usepackage{amsfonts}
\usepackage{latexsym}
\usepackage{amssymb}
\usepackage[cmtip,arrow]{xy}
\usepackage{pb-diagram,pb-xy}
\usepackage{pstricks}
\usepackage{pst-all}
\usepackage{enumerate}

\newcommand{\Z}{{\mathbf Z}}
\newcommand{\Q}{{\mathbf Q}}

\newcommand{\kk}{{\mathbf k}}
\newcommand{\cat}{{\sf {cat }}}

\newcommand{\id}{{\rm {id }}}

\newcommand{\tc}{{\rm\sf {TC}}}
\newcommand{\supp}{\rm {supp}}

\newcommand{\comment}[1]{}

\newcommand{\card}{\rm {cp}}

\def\U{\mathcal U}

\def\V{\mathcal V}

\def\zcl{\rm {zcl}}

\newtheorem{theorem}{Theorem}[section]

\newtheorem{lemma}[theorem]{Lemma}
\newtheorem{ex}[theorem]{Example}
\newtheorem{corollary}[theorem]{Corollary}
\newtheorem{definition}[theorem]{Definition}

\begin{document}


%

\title{CONFIGURATION SPACES AND ROBOT MOTION PLANNING ALGORITHMS}\footnote{To appear in the volume \lq\lq Combinatorial and Toric Topology\rq\rq, Lecture Note Series, Institute for Mathematical Sciences, National University of Singapore (\copyright © World Scientific Publishing Co., 2017 ) }
\maketitle

\markboth{M.Farber}{Motion planning algorithms}

\begin{center}{\author{Michael Farber}

\address{School of Mathematical Sciences\\
Queen Mary, University of London\\
Mile End Road, London E1 4NS, UK \\
m.farber@qmul.ac.uk}}\end{center}

\begin{abstract}
The paper surveys topological problems relevant to the motion planning problem of robotics and includes some new results and constructions. 
First we analyse the notion of topological complexity of configuration spaces which is responsible for discontinuities in algorithms for robot navigation. 
Then we present explicit motion planning algorithms for coordinated collision free control of many particles moving in Euclidean spaces or on graphs. These algorithms are optimal in the sense that they have minimal number of regions of continuity. Moreover, we describe in full detail the topology of configuration spaces of two particles on a tree and use it to construct some
top-dimensional cohomology classes in configuration spaces of $n$ particles on a tree. 
\end{abstract}

\vspace*{12pt}
%

\section{Introduction}

This paper starts  with a survey of the topological approach to the motion planning problem complementing \cite{F4} and chapter 4 of \cite{Finv}. 
 In \S \ref{sec2} -- \S \ref{sec7} we present a general description of the method and some basic results. 

In \S \ref{sec8} and \S \ref{sec9}  we analyse in full detail motion planning algorithms for collision free motion 
of many particles moving in the Euclidean spaces ${\Bbb R}^d$. Problems of this kind appear in many areas of engineering when multiple objects 
have to be moved in a coordinated way from one state to another avoiding collisions. The motion planning algorithms presented here are {\it optimal} in the sense that they have minimal topological complexity (equal $2n-1$ or $2n-2$ depending on the parity of the dimension $d$ where $n$ is the number of moving objects). 
The motion planning algorithms suggested in \cite{F4} had topological complexity quadratic in $n$. 
A recent paper \cite{HME} proposed a motion planning algorithm for $n$ particles moving on the plane ${\Bbb R}^2$ having complexity $2n-1$. 
The algorithms presented here are inspired by the construction of \cite{HME}.

In \S \ref{sec10} we analyse the topology of configuration spaces of graphs and present (following \cite{F3}) a motion planning algorithm for collision free control of $n$ particles on a tree. In \S \ref{secfg2} we describe explicitly the configuration space $F(\Gamma, 2)$ of two particles on a tree proving the main Theorem \ref{thm1} in full detail; this theorem was stated in \cite{F3} without proof. 

Theorem \ref{tctree} claims that the topological complexity of collision free motion of many particles on a tree is independent of the number of moving particles and depends only on the number of essential vertices of $\Gamma$. 
This fact contrasts the corresponding result for Euclidean spaces where the complexity is linear in $n$. The proof of Theorem \ref{tctree}
is completed in \S \ref{sec13} after an analysis of top-dimensional cohomology of configuration spaces of trees which is carried out in \S \ref{sec12}. 
Theorem \ref{tctree} was stated in \cite{F3} without proof. Recently S. Scheirer \cite{SS} published a detailed proof of a similar result under some additional assumptions. 

In \S \ref{sec14} we make some further comments and most recent literature references.

%
%
%
%
%
%

\section{Motion planning algorithms and topological complexity of configuration spaces}\label{sec2}

\subsection{Motion planning algorithms} Any mechanical system $S$ possesses a variety of states $C(S)$, called {\it the configuration space}. 
As an example, one may imagine a robot whose state consists of its location in the 3-space as well as the mutual positions of all its body parts such as elbows, knees, fingers etc. 

We want to programme our system $S$ so that it is capable of moving {\it autonomously} from any initial state $A\in C(S)$ to any final state $B\in C(S)$. Such programme is {\it a motion planning algorithm}. 
Once a motion planning algorithm has been specified, we may simply {\it order} our system
to move to a new state $B$, and the motion planning algorithm will prescribe {\it how} the system will implement the motion departing from the current state $A$. 

A state of the system is typically described by a collection of numerical parameters which can be interpreted as coordinates of a point in ${\Bbb R}^N$. The variety of all states of the system is then represented by a subset 
 $C(S)\subset {\Bbb R}^N$; we see that the configuration space of the system comes naturally with a topology. 
 The topology of 
the configuration space $C(S)$ is important since motions of the system are represented by {\it continuous} paths in $C(S)$. 

We refer to  \cite{L} and \cite{Sh} for additional information about motion planning in robotics.

\subsection{The concept of $\tc(X)$} We shall study a topological invariant $\tc (X)$ of a topological space $X$, originally introduced in \cite{F1}, see also \cite{F2} and \cite{F4}. It is a numerical homotopy invariant
inspired by the robot motion problem, similar in spirit to the classical Lusternik -- Schnirelmann category $\cat(X)$.
Intuitively, $\tc(X)$ is a measure of the navigational complexity of $X$ viewed as the configuration space of a system. 
$\tc(X)$, as well as $\cat(X)$, are special cases of a more general notion of the genus of a fibration introduced by A. Schwarz \cite{Sz}.



Next we give the formal definitions. Let $X$ denote a topological space though of as the configuration space of a mechanical system. The states of the system are represented by the points of $X$, and continuous
motions of the system are represented by continuous paths $\gamma: [0,1]\to X$.
Here the point $A=\gamma(0)$ represents the initial state and $\gamma(1)=B$
represents the final state of the system. The space $X$ is path connected if and only if
the system can be brought to an arbitrary state from any given state by a continuous motion. 

Denote by $PX=X^I$ the space of all continuous paths $\gamma: I= [0,1]\to X.$ The
space $PX$ is supplied with the compact--open topology, see  \cite{Sp}, which is characterised by the property that a map $Z\to X^I$ is continuous if and only if the associated map $Z\times I\to X$ is continuous. Let
\begin{eqnarray}\label{fibration}
\pi: PX\to X\times X
\end{eqnarray}
be the map which assigns to a path $\gamma$ the pair $(\gamma(0), \gamma(1))\in
X\times X$ of the initial -- final configurations. It is easy to see that $\pi$ is a fibration in the
sense of Serre, see \cite{Sp}, chapter 2, \S 8, Corollary 3.

{\it A motion planning algorithm} is a section of the fibration
$\pi$. 
 In other words a motion planning algorithm is a map (not necessarily continuous)
$$ s: X\times
X\to PX$$ satisfying
$
\pi\circ s= 1_{X\times X}.$

A motion planning algorithm $s: X\times X\to PX$ is {\it continuous} if the suggested route $s(A,B)$ of going from $A$ to $B$ depends continuously on
the states $A$ and $B$.
A continuous motion planning algorithm in $X$ exists if and only if the space
$X$ is contractible, see \cite{Finv}, Lemma 4.2.
Thus, for a system with non-contractible configuration space any motion
planning algorithm must be discontinuous.

\begin{definition}\label{def1}
Given a path-connected topological space $X$, we define the topological complexity of $X$ as the minimal number 
$\tc(X)=k$ such that the Cartesian product $X \times X$ may be covered by $k$ open subsets
$X \times X = U_1 \cup U_2 \cup \dots U_k$ such that for any $i = 1,2,...,k$ there exists a continuous section
$s_i: U_i \to PX,$
$\pi\circ s_i=\id$ over $U_i$. If no such $k$ exists we will set $\tc(X)=\infty$.
\end{definition}

\begin{ex} \label{ex22} {\rm Suppose we are to construct a motion planning algorithm on the circle $X=S^1$. Given two points 
$A, B\in S^1$, which are not antipodal, i.e. $B\not= -A$, we may move from $A$ to $B$ along the shortest geodesic curve $s_1(A, B)$ 
which is unique and depends continuously on $A$ and $B$. 
This defines a continuous section $s_1:U_1\to (S^1)^I$, where $U_1\subset S^1\times S^1$ denotes the set 
$U_1=\{(A, B)\in S^1\times S^1; A\not=-B\}$.

However, if the points $A$ and $B$ are antipodal then there are two distinct shortest geodesic curves from $A$ to $B$ so that the section $s_1$ does not extend to a continuous section over the whole product $S^1\times S^1$. 

Denote $U_2=\{(A, B)\in S^1\times S^1; A\not=B\}\subset S^1\times S^1$. We may define a continuous section $s_2:U_2\to (S^1)^I$ by setting 
$s_2(A, B)$ to be the path moving from $A$ to $B$ in the clockwise direction along the circle with constant velocity. Again, we observe that the section $s_2$ cannot be 
extended to a continuous section on the whole space $S^1\times S^1$. 

The open sets $U_1, U_2$ cover $S^1\times S^1$ and therefore $\tc(S^1)\le 2$ according to Definition \ref{def1}. 
On the other hand, since the circle $S^1$ is not contractible we know that $\tc(S^1)>1$. Therefore $\tc(S^1)=2.$} 
\end{ex}

\subsection{Homotopy invariance} Next we show that the topological complexity $\tc(X)$ depends only on the homotopy type of $X$. 
We start with the following auxiliary statement. 

\begin{theorem}\label{domination} Let $X$ and $Y$ be topological spaces. Suppose that $X$ dominates $Y$ , i.e., there exist continuous maps $f : X \to Y$ and $g: Y \to X$ such that $f \circ g \sim \id_Y$. Then $\tc(Y) \le \tc(X).$
\end{theorem}
\begin{proof} Assume that $U \subset X \times X$ is an open subset such that there exists a continuous section 
$s: U \to PX$ of (\ref{fibration}) over $U$. Define $V = (g \times g)^{-1}(U) \subset Y \times Y.$ 
We may construct a continuous section $\sigma: V \to PY$ over $V$ as follows. 
Fix a homotopy $h_t: Y \to Y$ with $h_0=\id_Y$ and $h_1= f\circ g$; here $t\in [0,1]$. For $(A,B)\in V $ and $t\in [0,1]$ set
$$
\sigma(A, B) (\tau) = \left\{
\begin{array}{ll} 
h_{3\tau}(A), & \mbox{for} \quad 0\le \tau\le 1/3,\\ \\
f(s(gA,gB)(3\tau-1)), & \mbox{for} \quad 1/3\le \tau\le 2/3,\\ \\
h_{3(1-\tau)}(B), & \mbox{for}  \quad 2/3 \le \tau\le 1. 
\end{array}
\right.
$$
Thus we obtain that for $k = \tc(X)$, any open cover $U_1 \cup \dots \cup U_k = X \times X$ with a continuous section $s_i:U_i\to PX$ over each $U_i$ 
defines an open cover $V_1\cup \dots \cup V_k$ of $Y \times  Y$, where each $V_i= (g\times g)^{-1}(U_i)$ admits a continuous section 
$\sigma_i: V_i\to PY$. 
This proves that $\tc(Y)\le \tc(X)$. 
\end{proof}

\begin{corollary}
If $X$ and $Y$ are homotopy equivalent then $\tc(X)=\tc(Y)$. 
\end{corollary}

\section{Upper and lower bounds for $\tc(X)$}

\subsection{The upper bound} We start with a dimensional upper bound. 
\begin{theorem}\label{cor1} For any path-connected paracompact locally contractible topological space $X$ one has
\begin{eqnarray}
\tc(X)\, \leq \,  \dim(X\times X) +1.\label{ineq1}
\end{eqnarray}
\end{theorem}

Here $\dim(X\times X)$ denotes the covering dimension of $X\times X$. 

Recall that $\dim (Y)\leq n$ if any 
open cover of $Y$ has a locally finite open refinement such that no point of $Y$ belongs 
to more than $n+1$ open sets of the refinement. If $Y$ is a polyhedron, then $\dim(Y)$ coincides with the 
maximum of the dimensions of the simplices of $Y$.

A topological space $Y$ is called {\it locally contractible} if any point of $Y$ has an open neighbourhood $U\subset Y$ 
such that the inclusion $U\to Y$ is null-homotopic. 

\begin{proof} Denote $\dim (X\times X)=n$. Let $\U=\{U_i\}_{i\in I}$ be an open cover of $X\times X$ such that each open set $U_i\subset X\times X$ admits a continuous section 
$s_i:U_i\to PX$, where $i\in I$. Such cover exists since $X$ is locally contractible. Let $\V=\{V_j\}_{j\in J}$ be a refinement of $\U$ having multiplicity $\le n+1$, i.e. for every $(x, y)\in X\times X$ there exist at most $n+1$ values of $j\in J$ such that $(x, y)\in V_j$. 
Construct a partition of unity $\{h_j\}_{j\in J}$ subordinate to $\V$, i.e. each $h_j:X\times X\to [0,1]$ is continuous, ${\rm {supp}}(h_j) \subset V_j$ and $\sum_{j\in J}h_j = 1_{X\times X}$. Given a subset 
$S\subset J$ define $$W(S)\subset X\times X$$ as the set of all pairs $(x, y)$ such that $h_j(x, y)>h_k(x, y)$ for all $j, k\in J$ satisfying $j\in S$ and 
$k\notin S$. Each set $W(S)$ is open and admits a continuous section $W(S)\to PX$. 
Besides,  
$W(S)=\emptyset$ for $|S|>n+1$ and 
the family $\{W(S); |S|\le n+1\}$ is an open cover  of $X\times X$. 
If $S, S'\subset J$ are two subsets such that none of them contains the other, i.e. there is $j\in S$, $j\notin S'$ and there exists $k\in S'$, $k\notin S$, then the intersection $W(S)\cap W(S')=\emptyset$ is empty. Therefore the union 
$$W_k=\bigcup_{|S|=k}W(S),$$
is open and admits a continuous section $W_k\to PX$, where $k=1, 2, \dots, n+1$.
We obtain an open cover $\{W_1, \dots, W_{n+1}\}$ of $X\times X$ with the desired properties implying that $\tc(X) \le n+1$. 
\end{proof}

\subsection{The lower bound} Next we give a lower bound for $\tc(X)$ which depends on the structure of the cohomology algebra of $X$.  

Let $\kk$ be a field. The singular cohomology $H^\ast(X;\kk)$ is a graded $\kk$-algebra with the multiplication
\begin{eqnarray}
\cup: H^\ast(X;\kk)\otimes H^\ast(X;\kk)\to H^\ast(X;\kk)\label{prod}
\end{eqnarray}
given by the cup-product, see \cite{Hat02}. For two cohomology classes $u \in H^i(X;\kk)$ and $v \in H^j(X;\kk)$ 
we shall denote their cup-product by $$u\cup v= uv\, \in \, H^{i+j}(X;\kk).$$
The tensor product $H^\ast(X;\kk)\otimes H^\ast(X;\kk)$ is also a graded $\kk$-algebra
with the multiplication 
\begin{eqnarray}\label{signs}
(u_1\otimes v_1)\cdot (u_2\otimes v_2) = (-1)^{|v_1|\cdot |u_2|}\, u_1u_2\otimes v_1v_2.
\end{eqnarray}
Here $|v_1|$ and $|u_2|$ denote the degrees of cohomology classes $v_1$ and $u_2$ correspondingly.
The cup-product (\ref{prod}) is an algebra homomorphism.

\begin{definition} The kernel of homomorphism (\ref{prod}) is called {\it the ideal of the zero-divisors} of $H^\ast(X;\kk)$.
The {\it zero-divisors-cup-length} of $H^\ast(X;\kk)$ is the length of the longest nontrivial product  under the multiplication (\ref{signs}) 
in the ideal of the zero-divisors
of $H^\ast(X;\kk)$.
\end{definition}

\begin{theorem}\label{thm4} The topological complexity of motion planning $\tc(X)$ 
is greater than the zero-divisors-cup-length of 
$H^\ast(X;\kk)$.
\end{theorem}

\noindent
\begin{proof} Let $\Delta_X\subset X\times X$ denote the diagonal. 
First we observe that the kernel of the induced homomorphism 
$\pi^\ast: H^j(X\times X; \kk)\to H^j(PX;\kk)$ coincides with the set of cohomology classes 
$u\in H^j(X\times X;\kk)$ such that $$u|_{\Delta_X}=0\in H^j(X;\kk).$$ Let $\alpha: X\to PX$ be the map which associates to any point $x\in X$ the constant path $[0,1]\to X$ at this point. Note that $\alpha$ is a homotopy equivalence and the composition $\pi\circ \alpha: X \to X\times X$ is the inclusion onto the diagonal $\Delta_X$ and thus our statement follows. 

%
Next we note that the composition
\begin{eqnarray*}
H^\ast(X;\kk)\otimes H^\ast(X;\kk) \simeq H^\ast(X\times X;\kk)\stackrel {\pi^\ast}\to H^\ast(PX;\kk)
\underset{\simeq}{\stackrel {\alpha^\ast}\to} H^\ast(X;\kk)
\end{eqnarray*} 
coincides with
the cup-product homomorphism (\ref{prod}) where the homomorphism on the left is the K\"unneth isomorphism. 

Combining these two remarks we obtain that a cohomology class $$u=\sum_r \, a_r\times b_r\in H^\ast(X\times X;\kk)$$ satisfies $\pi^\ast u =0$ if and only if the tensor 
$$\sum_r a_r\otimes b_r\in H^\ast(X;\kk)\otimes H^\ast(X;\kk)$$ is a zero-divisor. 

Suppose that $u_1, \dots, u_s\in H^\ast(X\times X;\kk)$ are cohomology classes satisfying $\pi^\ast(u_j)=0$ for $j=1, \dots, s$ and such that their cup-product 
$$0\not= u_1\cup \dots \cup u_s\in H^\ast(X\times X;\kk)$$ is nonzero. We claim that the topological complexity $\tc(X)$ must satisfy $\tc(X)\ge s+1$. Indeed, suppose that $\tc(X)\le s$, i.e. one may find an open cover $U_1, \dots, U_s$ of $X\times X$ with each open set $U_i$ admitting a continuous section
$\sigma_i: U_i\to PX$. We have 
$$u_i|_{U_i}= \sigma_i^\ast\circ \pi^\ast(u_i) =0$$
and from the exact long cohomology sequence of the pair $(X\times X, U_i)$ one obtains that there exists a relative cohomology class 
$\tilde u_i\in H^\ast(X\times X, U_i;\kk)$ such that 
\begin{eqnarray}
u_i = \tilde u_i|_{\Delta_X}, \quad i=1, \dots, s.
\end{eqnarray}
Thus we see that the nontrivial product $u_1u_2\dots u_s$ equals $(\tilde u_1\tilde u_2\dots \tilde u_s)|_{\Delta_X}$; however the product 
$\tilde u_1\tilde u_2\dots \tilde u_s$ lies in the trivial group $H^\ast(X\times X, \cup_i U_i;\kk)= H^\ast(X\times X, X\times X;\kk)=0$ contradicting
our assumption $u_1u_2\dots u_s\not=0$. 
%
\end{proof}

\begin{ex}\label{ex5}{\rm  Let $X=S^n$. Let $u\in H^n(S^n;\kk)$
be the fundamental class, and let $1\in H^0(S^n;\kk)$ be the unit.
Then the class $a=1\otimes u-u\otimes 1\in H^\ast(S^n;\kk)\otimes H^\ast(S^n;\kk)$ is a zero-divisor, since
applying the homomorphism (\ref{prod}) to it we obtain 
$ 1\cdot u-u\cdot 1=0.$ 
Another zero-divisor is $b=u\otimes u$, since $u^2=0$. 
Computing $a^2=a\cdot a$ by means of rule (\ref{signs}) we find
$$a^2=((-1)^{n-1}-1)\cdot u\otimes u.$$ 
Hence $a^2=-2 b$ for $n$ even and $a^2=0$ for $n$ odd;
the product $ab$ vanishes for any $n$.
We conclude that {\it the zero-divisors-cup-length of $H^\ast(S^n;\Q)$
is greater or equal than 1 for $n$ odd and is greater or equal than 2 for $n$ even.} }\end{ex}

Applying Theorem \ref{thm4}
we find that $\tc(S^n)>1$ for $n$ odd and $\tc(S^n)>2$ for $n$ even.
This means that any motion planner on the sphere $S^n$ must have at least two open sets $U_i$; 
moreover, if $n$ is even, any motion planner on the sphere $S^n$ must have at least 
three open sets $U_i$.

\begin{ex} \label{exgraph} Let $X$ be a connected finite graph with $b_1(X)>1$. Then there exist two linearly independent cohomology classes 
$u_1, u_2\in H^1(X;\Q)$. Then for $i=1, 2$ the tensors $1\otimes u_i-u_i\otimes 1$ are zero-divisors and their product equals 
$u_2\otimes u_1-u_1\otimes u_2\not=0$. Hence by Theorem \ref{thm4} we have $\tc(X)\ge 3$. On the other hand, applying Theorem \ref{cor1} we obtain 
$\tc(X)\le 3$. Therefore, $\tc(X)=3$.

\end{ex}

\section{Simultaneous control of several objects}

Suppose that we have a system which is  a union of two independent systems $S_1$ and $S_2$ such that $S_1$ and $S_2$ can move independently without interaction. For example one may imagine the situation that an operator has to control two robots confined to two different rooms in the house
simultaneously. If $X_i$ denotes the configuration space of the system $S_i$, where $i=1, 2$, then the configuration space of our entire system is the Cartesian product $X_1\times X_2$, the variety of all pairs of states $(x_1, x_2)$ where $x_1\in X_1$ and $x_2\in X_2$. 

Note that in the case of two robots operating in the same room we would have to exclude from the product $X_1\times X_2$ the set of all pairs of configurations $(x_1, x_2)$ where the robots collide; thus, in this case the actual configuration space will be a suitable subspace of the product $X_1\times X_2$.

\subsection{The product inequality}
\begin{theorem}\label{prod11} For path--connected metric spaces
$X$ and $Y$ one has
\begin{eqnarray}
\tc(X\times Y)\leq \tc(X) +\tc(Y) -1.
\end{eqnarray}
\end{theorem}
\begin{proof} Denote $\tc(X)=n$, $\tc(Y)=m$. Let $U_1,\dots, U_n$ be an open cover of $X\times X$ with a continuous section $s_i: U_i\to PX$ for $i=1, \dots, n$.  Let $f_i: X\times X\to {\Bbb R}$, where $i=1, \dots, n$, be a partition of unity
subordinate to the cover $\{U_i\}$.
Similarly, 
let $V_1,\dots, V_m$ be an open cover of $Y\times Y$ with a continuous section $\sigma_j: V_j\to PY$ for $j=1, \dots, m$, and let $g_j: Y\times Y\to {\Bbb R}$, where $j=1, \dots, m$, be a partition of unity subordinate to the cover $\{V_j\}$. 

For any pair of nonempty subsets $S\subset \{1, \dots, n\}$ and $T\subset \{1, \dots, m\}$, let 
$$W(S,T)\subset (X\times Y)\times (X\times Y)$$ 
denote the set of all
4-tuples $(A,B, C,D)\in (X\times Y)\times (X\times Y)$, such that for any $(i,j)\in S\times T$ 
it holds that
$$f_i(A,C) \cdot g_j(B,D) > 0,$$
and for any $(i',j')\notin S\times T$, 
$$f_i(A,C) \cdot g_j(B,D) > f_{i'}(A,C) \cdot g_{j'}(B,D).$$
One easily checks that:

{\it (a) each set $W(S,T)\subset  (X\times Y)\times (X\times Y)$ is open;

(b) $W(S,T)$ and $W(S', T')$ are 
disjoint if neither $S\times T\subset S'\times T'$ nor $S'\times T'\subset S\times T$;

(c) if $(i,j)\in S\times T$, then $W(S,T)$ is contained in $U_i\times V_j$; therefore there exists a continuous motion planning algorithm
over each $W(S,T)$ (it can be described explicitly in terms of $s_i$ and $\sigma_j$);

(d) the sets $W(S,T)$ (with all possible nonempty $S$ and $T$) form a cover of $(X\times Y)\times (X\times Y)$.}

Let us prove (d). Suppose that $(A,B,C,D)\in (X\times Y)\times (X\times Y)$. Let $S$ be the set of all indices 
$i\in \{1, \dots, n\}$,
such that $f_i(A,C)$ equals the maximum of $f_k(A,C)$, where $k=1, 2, \dots, n$. 
Similarly, let $T$ be the set of all $j\in \{1, \dots, m\}$, such that $g_j(B,D)$ equals the maximum of $g_\ell(B,C)$, where
$\ell=1, \dots, m$. Then clearly $(A,B,C,D)$ belongs to $W(S,T)$.

Let $W_k \subset (X\times Y)\times (X\times Y)$ 
denote the union of all sets $W(S,T)$, where $|S|+|T|=k$. Here $k=2, 3, \dots, n+m.$
The sets $W_2, \dots, W_{n+m}$ form an open cover of $(X\times Y)\times (X\times Y)$. 
If $|S|+|T|=|S'|+|T|=k,$ then the corresponding sets $W(S,T)$ and $W(S',T')$ either coincide (if $S=S'$ and $T=T'$) or are disjoint. 
Hence we see (using (c)) that there exists a continuous motion planning algorithm over each open set $W_k$. 
This completes the proof.
\end{proof}

\subsection{The reduced topological complexity} Theorem \ref{prod11} suggests the notation
$$\widetilde\tc(X)=\tc(X)-1,$$
which is called {\it the reduced topological complexity. } Then we have:

\begin{corollary}\label{prod22}
For path-connected metric spaces $X_1, \dots, X_k$ one has 
\begin{eqnarray}\label{prod8}
\widetilde\tc(X_1\times X_2\times \dots\times X_k)\le \sum_{i=1}^k \widetilde\tc(X_i).
\end{eqnarray}
\end{corollary}

We shall use the following notation. For a topological space $X$ we shall denote by $\zcl (X)$ the largest integer $k$ such that 
that there exist $k$ zero-divisors $u_1, u_2, \dots, u_k\in  H^\ast(X; \Q) \otimes H^\ast(X;\Q)$ 
having a nontrivial product $$u_1u_2 . . . u_k \not= 0 \in  H^\ast(X; \Q) \otimes H^\ast(X;\Q).$$
Theorem \ref{thm4} can be restated as the inequality 
$$\widetilde\tc(X) \ge \zcl(X).$$
By Example \ref{ex5} we have
$$\zcl(S^n) \ge  \left \{
\begin{array}{ll}
2, & \mbox{if $n$ is even},\\ \\ 

1, & \mbox{if $n$ is odd}.\end{array}
\right.
$$

\begin{lemma} \label{lm7} One has $\zcl(X\times Y) \ge \zcl(X) + \zcl(Y)$.
\end{lemma}
\begin{proof} See \cite{Finv}, Lemma 4.52.

\end{proof}

%
%

\begin{ex} {\rm Suppose that each space $X_i$ is the $n$-dimensional sphere $S^n$. Then using Corollary \ref{prod22} and Lemma \ref{lm7} one has

$$
\tc(\large\prod_{i=1}^k S^n) 
\le 
\left\{
\begin{array}{ll}
2k +1, & \mbox{for $n$ even}, \\ \\
k +1, & \mbox{for $n$ odd}.
\end{array}
\right.
$$
On the other hand we have
\begin{eqnarray*}
\widetilde\tc(\prod_{i=1}^k S^n) &\ge& \zcl(\prod_{i=1}^k S^n)\\
&\ge& \sum_{i=1}^k \zcl(S^n)\\
&=& \left\{
\begin{array}{ll}
2k, & \mbox{if $n$ is even},\\ \\ 
k, & \mbox{if $n$ is odd}.
\end{array}
\right. 
\end{eqnarray*} 

Thus:
\begin{eqnarray}\label{prodspheres}
\tc(\large\prod_{i=1}^k S^n) 
= 
\left\{
\begin{array}{ll}
2k +1, & \mbox{for $n$ even}, \\ \\
k +1, & \mbox{for $n$ odd}.
\end{array}
\right.
\end{eqnarray}
}
\end{ex}

\section{Centralised and distributed controls for large systems}

Consider a large system $S$ consisting of many independently moving parts $S_1, \dots, S_k$. 
As we discussed earlier, the configuration space of this system is the Cartesian product
$X_1\times X_2\times \dots\times X_k$ of the configuration spaces $X_i$ of individual parts $S_i$. 
One may compare the distributed and centralised motion planning algorithms for $S$. 

In the case of {\it distributed motion planning algorithms}, one controls each system $S_i$ independently of the other systems $S_j$. The motion planning algorithm for $S_i$ will have at least $\tc(X_i)$ domains of continuity, and therefore a distributed motion planning algorithm for $S$ will have at least 
\begin{eqnarray}
\large\prod_{i=1}^k \tc(X_i)\end{eqnarray}
of domains of continuity. 

However, in the case of {\it centralised control}, when the system $S$ is viewed as  a single system, there exists a motion planning algorithm with 
$$\tc(\large \prod_{i+1}^k X_i)$$ domains of continuity. Taking into account inequality (\ref{prod8}) we obtain that one may find a centralised motion planning algorithm for $S$ having at most 
\begin{eqnarray}\label{sum4}
1-k + \sum_{i=1}^k \tc(X_i) = 1+ \sum_{i+1}^k \widetilde\tc(X_i)
\end{eqnarray}
domains of continuity. 

In the special case when $\tc(X_i)=a\ge 2$ is independent of $i$, we obtain that any distributed motion planning algorithm has at least $a^k$ domains of continuity and one can find a centralised motion planning algorithm with at most $k(a-1)+1$ domains of continuity. 

In conclusion, {\it the centralised control has potentially significantly more stability compared to the distributed control}.

\section{Tame motion planning algorithms}

\subsection{}  The definition of $\tc(X)$ (see Definition \ref{def1}) deals with open subsets of $X\times X$ admitting continuous sections 
of the path fibration (\ref{fibration}). To construct a motion planning algorithm in practice 
one partitions the whole space $X\times X$ into pieces and defines 
a continuous (often smooth or analytic) section over each of the obtained sets. 
Any such partition necessarily contains sets which are not open and 
 hence we need to be able to operate with subsets of $X\times X$  of more general nature.


\begin{definition}
A topological space $X$ is an Euclidean Neighbourhood Retract (ENR) if it can be embedded into an Euclidean space ${\Bbb R}^k$ such that for some open 
neighbourhood $X\subset U\subset {\Bbb R}^k$ there is a retraction $r: U\to X$, $r|_X=1_X$. 
\end{definition}
It is known that a subset $X \subset {\Bbb R}^k$ is an ENR if and only if it is locally compact and locally contractible, see \cite{Dold}, Chapter 4, §8.
This implies that all finite-dimensional polyhedra, manifolds and semi-algebraic sets are ENRs.

\begin{definition}
Let $X$ be an ENR. A motion planning algorithm $s: X\times X\to PX$ is said to be tame if $X\times X$ can be split into finitely many sets 
$$X\times X= F_1\cup F_2\cup \dots\cup F_k$$
such that

1. Each restriction $s|{F_i}: F_i\to PX$ is continuous, where $i=1, \dots, k$; 

2. $F_i\cap F_j=\emptyset$ for $i\not= j$; 

3. Each $F_i$ is an ENR. 
\end{definition}

It is known that for an ENR $X$, the minimal number of domains of continuity $F_1, \dots, F_k$ in tame motion planning algorithms 
$s: X\times X\to PX$ equals $\tc(X)$, see \cite{F4}, Theorem 13.1.

\begin{ex} {\rm Here we construct a tame motion planning algorithm on the sphere $S^n$. 
Let $F_1\subset S^n\times S^n$ be the set of all pairs $(A, B)$ such that
$A\not=-B$. We may construct a continuous section $s_1: F_1\to PS^n$ by moving
$A$ towards $B$ along the shortest geodesic arc.
Consider now the set $F_2\subset S^n\times S^n$ of all pairs
antipodal points $(A,-A)$. If $n$ is {odd} we may construct a continuous section
$s_2: F_2\to PS^n$ as follows. Fix a non-vanishing tangent vector field $v$ on
$S^n$; such $v$ exists for $n$ odd. Move $A$ towards the antipodal point $-A$ along the semi-circle tangent
to vector $v(A)$.

In the case when $n$ is even the above procedure has to be modified since for $n$ even any vector field $v$ tangent to $S^n$ has at least one zero.
We may find a tangent vector field $v$ having a single zero
$A_0\in S^n$. Denote $F_2=\{(A, -A); A\not=A_0\}$ and define $s_2: F_2\to PS^n$
as in the previous paragraph. The set $F_3=\{(A_0, -A_0)\}$ consists of a single pair; define
$s_3:F_3\to PS^n$ by choosing an arbitrary path from $A_0$ to $-A_0$.
}
\end{ex}

\begin{ex}\label{exdef} {\rm 
 {\bf Constructing sections via deformations.} Let $A\subset X\times X$ be a subset. A continuous section 
 $$s_A: A\to X^I$$ of the fibration (\ref{fibration}) can be viewed as  
 a continuous function of three variables $s_A(a_1, a_2, t)\in X$, where $a_1, a_2\in X$ are such that $(a_1, a_2)\in A$  and $t\in I=[0,1]$; this function must satisfy $s_A(a_1, a_2, 0)=a_1$ and $s_A(a_1, a_2, 1)=a_2$. 
 
 Suppose that a subset $B\subset X\times X$ can be continuously deformed inside $X\times X$ into the subset $A$. In other words, 
 assume that there exists a continuous map 
 \begin{eqnarray}\label{deformation}
 h: B\times I\to X\times X
 \end{eqnarray} 
 such that $h(b, 0)=b$ and $h(b, 1)\in A$ for any $b\in B$. 
We may write 
$$h(b, t)=(h_1(b, t), h_2(b, t))$$ 
where $h_1$ and $h_2$ are the compositions of $h$ with the projections. 
The path $s_A(h(b, 1), \tau)$, where $\tau\in [0,1]$, connects the points $h_1(b,1)$ and $h_2(b,1)$, i.e. $s_A(h(b, 1), 0)=h_1(b, 1)$
 and  $s_A(h(b, 1), 1)=h_2(b, 1)$. Thus, the formula
 \begin{eqnarray}\label{concatenation}
 s_B(b, \tau) =\, \left\{
 \begin{array}{lll}
 h_1(b, 3\tau) & \mbox{for}& \tau\in [0, 1/3],\\ \\
 s_A(h(b,1), 3\tau-1) & \mbox{for}& \tau\in [0, 1/3],\\ \\
h_2(b, 3-3\tau), & \mbox{for}& \tau\in [2/3, 1]
\end{array}
 \right.
 \end{eqnarray}
 defines a path from $b_1$ to $b_2$ which depends continuously on $(b, t)=(b_1, b_2, t)$. Hence we obtain a continuous 
 section $$s_B: B\to X^I$$ of the path fibration (\ref{fibration}) over $B$. Thus a deformation of $B$ into $A$ and a continuous section $s_A$ over $A$ define a continuous section $s_B$ over $B$.

 We shall often use the above remark in the case when $A$ is the diagonal $\Delta_X\subset X\times X$, i.e. $\Delta_X=\{(x, y)\in X\times X; x=y\}$. 
 There exists an obvious section $s: \Delta_X\to X^I$ over the diagonal and hence any deformation (\ref{deformation}) 
 of a subset $B\subset X\times X$ into the diagonal $\Delta_X$
 will automatically give a motion planning section over $B$, using (\ref{concatenation}).

%
%
}

\end{ex}

\section{Topological complexity of collision free motion planning in ${\Bbb R}^d$}\label{sec7}

Consider a system consisting of $n$ small objects moving in the Euclidean space ${\Bbb R}^d$ without collisions. Mathematically we may assume that each of the objects is a point and hence our configuration space is 
$$F({\Bbb R}^d,n)=\{(z_1, \dots, z_n)\in {\Bbb R}^d\times \dots\times {\Bbb R}^d; z_i\not=z_j \quad \mbox{for}\quad i\not= j\}.$$
Here the vectors $z_i\in {\Bbb R}^d$ represent the locations of the moving objects and the condition $z_i\not= z_j$ reflects the requirement that the objects must move without collisions. 

A motion planning algorithm in $F({\Bbb R}^d, n)$ assigns to any pair of configurations 
$$(z_1, \dots, z_n), \, (z'_1, \dots, z'_n)\in F({\Bbb R}^d, n)$$ a continuous curve of configurations 
$$(z_1(t), z_2(t), \dots, z_n(t))\in F({\Bbb R}^d, n),\quad t\in [0,1],$$ such that 
$(z_1(0), \dots, z_n(0))=(z_1, \dots, z_n)$ and $(z_1(1), \dots, z_n(1))=(z'_1, \dots, z'_n)$.

The following theorem gives the topological complexity of this motion planning problem:

\begin{theorem}\label{thm11}  See \cite{FY}, \cite{FG5}.
For $n\ge 2$, one has
\begin{eqnarray}
\tc(F({\Bbb R}^d, n)) = \left\{
\begin{array}{ll}
2n-1, & \mbox{for $d\ge 3 $ odd}, \\ \\
2n-2, &\mbox{for $d\ge 2$ even}.
\end{array}
\right.
\end{eqnarray}
\end{theorem}

We see that the topological complexity of collision free motion planning in the Euclidean space is roughly $\sim 2n$ where $n$ is the number of controlled objects. 
One naturally wants to know explicit motion planning algorithms for $F({\Bbb R}^d,n)$ with optimal topological complexity as given by Theorem \ref{thm11}.
Such algorithms will be given in the following two sections. 

The paper \cite{F4} suggested a motion planning algorithm in $F({\Bbb R}^d,n)$ having topological complexity quadratic in $n$. 
In \cite{HME}, Hugo Mas-Ku and Enrique Torres-Giese suggested a motion planning algorithm for $F({\Bbb R}^2, n)$ with complexity $2n-1$. 
They also briefly indicated how their algorithm may work for $F({\Bbb R}^d, n)$ with any $d\ge 2$; such an algorithm would 
be optimal for $d\ge3$ odd. 


The paper \cite{HME} also contains the useful observation that the lower bound of Theorem \ref{thm11} for the case $d\ge 3$ 
odd can be deduced from the fact that the configuration space
 $F({\Bbb R}^d, n)$ contains the product $\prod_{i=1}^{n-1} S^{d-1}$ of $n-1$ copies of the sphere $S^{d-1}$ as a retract. 
 We repeat this argument below.
 
 Let us describe the maps
\begin{eqnarray}
\prod_{i=1}^{n-1} S^{d-1} \, \stackrel\alpha\to\,  F({\Bbb R}^d, n)\, \stackrel \beta\to\,  \prod_{i=1}^{n-1} S^{d-1};
\end{eqnarray}
we want  $\beta$ to be a retraction on the image of $\alpha$. 
 We think of $S^{d-1}\subset {\Bbb R}^d$ as being the unit sphere with centre at the origin. For 
$$(u_1, \dots, u_{n-1})\in \prod_{i=1}^{n-1} S^{d-1}$$ we set
$$\alpha(u_1, u_2, \dots, u_{n-1}) = (z_1, z_2, \dots, z_n)\in F({\Bbb R}^d, n)$$ where $z_1=0$ and for $i=1, \dots, n-1$ one has
$$ z_{i+1} =z_i +3^{i-1} u_i.$$
Clearly, $\alpha$ is injective. We have for $k\ge 1$, 
$$z_{i+k} - z_i = 3^{i+k-2}u_{i+k-1} + 3^{i+k-3}u_{i+k-2}+\dots +3^{i-1}u_{i}$$
and 
$$|z_{i+k} - z_i|\ge 3^{i+k-2}- 3^{i+k-3} - 3^{i+k-4}- \dots - 3^{i-1} = \frac{1}{2} 3^{i-1}(3^{k-1} +1). $$
This shows that $z_i\not= z_j$ for $i\not= j$, i.e. the map $\alpha$ indeed takes its values in the configuration space $F({\Bbb R}^d, n)$. 

Next we define the second map $\beta: F({\Bbb R}^d, n)\to \prod_{i=1}^{n-1} S^{d-1}$,
$$\beta(z_1, z_2, \dots, z_n) = (u_1, u_2, \dots, u_{n-1})$$ where 
$$u_i = \frac{z_{i+1} - z_i}{|z_{i+1} - z_i|}\in S^{d-1}, \quad \quad i=1, 2, \dots, n-1.$$
It is obvious that $\beta\circ \alpha$ is the identity map, i.e. $\beta$ is a retraction of the image of $\alpha$. 

Assuming that $d\ge 3$ is odd and $n\ge 2$, one applies Theorem \ref{domination} and formula (\ref{prodspheres}) to obtain 
$$\tc(F({\Bbb R}^d, n)) \ge \tc(\prod_{i=1}^{n-1} S^{d-1}) = 2n-1.$$

\section{A motion planning algorithm in $F({\Bbb R}^d, n)$}\label{sec8}

In this section we present a tame motion planning algorithm in $F({\Bbb R}^d, n)$ with $2n-1$ regions of continuity. 
This algorithm works for any $d\ge 2$ and $n\ge 2$; it is optimal in the case when $d\ge 3$ is odd. 
In the following section we suggest a modification of this algorithm which works in the case of 
$d\ge 2$ even and has $2n-2$ regions of continuity; this algorithm is optimal for $d\ge 2$ even in the sense that it has 
the smallest possible number of regions of continuity. 

The algorithm we describe in this section  can be used in designing 
practical systems controlling motion of many objects moving in space without collisions. 

\subsection{The space $F(L, n)$}\label{secfl} Fix a line $L\subset {\Bbb R}^d$ and denote by $p:{\Bbb R}^d\to L$ the orthogonal projection. 
Let $e\in {\Bbb R}^d$ denote a unit vector in the direction of $L$. The vector $e$ determines an order on $L$: we say that for $a, b\in L$ one has $a\le b$ if
the scalar product $\langle b-a, e\rangle $ is non-negative. 

Note that $F(L, n)$ is naturally a subset of $F({\Bbb R}^d, n)$ and while the configuration space $F({\Bbb R}^d, n)$ is connected, 
the configuration space $F(L, n)$ is disconnected. More precisely, the space $F(L, n)$ contains $n!$ connected components and each of the components of $F(L, n)$ is contractible. Indeed, a configuration $$C= (z_1, \dots, z_n)\in F(L, n)$$ determines a permutation 
\begin{eqnarray*}\label{perm}\tau= (i_1, \dots, i_n)\in \Sigma_n\end{eqnarray*} of the set $\{1, 2, \dots, n\}$ where
$$z_{i_1}<z_{i_2} < \dots< z_{i_n}.$$ This permutation describes the order in which the points $z_i$ appear on the line $L$. 
Clearly, two configurations 
of $F(L, n)$ lie in the same connected component of $F(L, n)$ if and only if they have the same ordering, i.e. they determine the same permutation. 

For a permutation $\tau\in \Sigma_n$ we denote by $F(L, n, \tau)$ the set of all configurations $C= (z_1, \dots, z_n)\in F(L, n)$ such that the order of points 
$z_i$ on the line $L$ is described by the permutation $\tau$. We have 
\begin{eqnarray}\label{split}
F(L, n) \, =\, \bigsqcup_{\tau\in \Sigma_n} F(L, n, \tau).\end{eqnarray}
To show that each space $F(L, n, \tau)$ is contractible  we note that for two configurations $C, C'\in F(L, n, \tau)$ where $C=(z_1, \dots, z_n)$ and
 $C'=(z'_1, \dots, z_n')$ we may define the linear deformation 
 \begin{eqnarray}\label{lindef}
 z_i(t)=(1-t) z_i+tz'_i, \quad t\in [0,1], \quad i=1, \dots, n,\end{eqnarray}
 which represents a continuous path in $F(L, n, \tau)$. Clearly, if $z_i<z_j$ and $z'_i< z'_j$ then $z_i(t)<z_j(t)$ for any $t\in [0,1]$. 
 In other words, we have a continuous section 
 \begin{eqnarray}\label{sectionpi}
 \sigma_\tau: F(L, n, \tau)\times F(L, n, \tau) \to F(L, n, \tau)^I\end{eqnarray}
 of the path fibration 
 $$F(L, n, \tau)^I \to F(L, n, \tau)\times F(L, n, \tau).$$
 
Fix a specific configuration $C_\tau\in F(L, n, \tau)$ for each permutation $\tau\in \Sigma_n$. Since $d>1$, the configurations $C_\tau$ and $C_{\tau'}$ can be connected by a continuous path
$$\gamma_{\tau, \tau'}: [0,1]\to F({\Bbb R}^d, n), \quad \gamma_{\tau, \tau'}(0) = C_\tau, \quad \gamma_{\tau, \tau'}(1) = C_{\tau'}.$$
The family of paths $\{\gamma_{\tau, \tau'}\}$ gives a continuous section of the path fibration over the (discrete) subset $$\Sigma_n\times \Sigma_n\subset F(L, n)\times F(L, n).$$
Since (\ref{lindef}) gives a deformation of the set $F(L, n, \tau)\times F(L, n, \tau')$ to the single point 
$\{(C_\tau, C_{\tau'})\} \subset \Sigma_n\times \Sigma_n$,
we obtain via concatenation (as explained in Example \ref{exdef}) 
 a continuous section
\begin{eqnarray}\label{sigmal}
\sigma: F(L, n)\times F(L, n) \to F({\Bbb R}^d, n)^I
\end{eqnarray}
of the path fibration, i.e. such that the composition 
\begin{eqnarray}\label{section1}
F(L, n)\times F(L, n)\stackrel \sigma \to F({\Bbb R}^d, n)^I \stackrel \pi\to F({\Bbb R}^d, n)\times F({\Bbb R}^d, n)
\end{eqnarray}
coincides with the inclusion 
$F(L, n)\times F(L, n)\subset  F({\Bbb R}^d, n)\times F({\Bbb R}^d, n)$. Recall that $\pi$ denotes the paths fibration (\ref{fibration}). 

\subsection{Sets $A_i$.} For a configuration 
$C\in F({\Bbb R}^d, n)$, where $C=(z_1, \dots, z_n)$ with $z_i\in {\Bbb R}^d$, $z_i\not=z_j$ for $i\not=j$, consider the set of projection points 
$$p(C)=\{p(z_1), \dots, p(z_n)\}, \quad \quad p(z_i)\in L, \quad i=1, \dots, n.$$
The cardinality of this set will be denoted $\card(C)$. Here the symbol \lq\lq $\card$\rq\rq stands for {\it \lq\lq cardinality of projection\rq\rq.}
Note that $\card(C)$ can be any number 
$1, 2, \dots, n$. 
Let $A_i$ denote the set of all configurations $C\in F({\Bbb R}^d, n)$ with $\card(C)=i$. Clearly, $A_i$ is an ENR. 

The set $A_n$ is open and dense in $F({\Bbb R}^d, n)$. If $C=(z_1, \dots, z_n)\in A_n$ then $p(C)\in F(L, n)$ and the formula
\begin{eqnarray}\label{projection}
z_i(t)=z_i+t (p(z_i)-z_i), \quad i=1, \dots, n
\end{eqnarray}
defines a continuous deformation of $A_n$ onto $F(L,n)$. 
%
%

In general, 
the closure of each set $A_i$ is contained in the union of the sets $A_j$ with $j\le i$, i.e. 
$$\overline A_i\subset \bigcup_{j\le i}A_j.$$

For a configuration 
$C\in A_i$, where $i\ge 2$, $C=(z_1, \dots, z_n)$ denote 
$$\epsilon(C)= \frac{1}{n} \min \{|p(z_r) -p(z_s)|; p(z_r)\not= p(z_s)\}.$$ 
For $i=1$ the formula above makes no sense and we set $\epsilon(C)=1$ for any $C\in A_1$. 

For  $C\in A_i$ and $t\in [0,1]$, where $C=(z_1, \dots, z_n)$, define
$$F_i(t)(C)=(z_1(t), \dots, z_n(t)), \quad\mbox{where}\quad z_j(t)=z_j+t(j-1)\epsilon(C)e,\quad j=1, \dots, n.$$
This defines a continuous deformation of $A_i$ into $A_n$ inside $ F({\Bbb R}^d, n)$; we shall call the deformation $F_i: A_i\times I\to F({\Bbb R}^d, n)$ 
 {\it \lq\lq desingularization\rq\rq. }

\subsection{Sections $\sigma_{ij}$.} 
We have constructed several deformations and a section over $F(L, n)$; applying iteratively the construction of Example \ref{exdef} 
we obtain a continuous section 
\begin{eqnarray}\label{sigmaij1}
\sigma_{ij}: A_i \times A_j\to F({\Bbb R}^d, n)^I, \quad i, j = 1, 2, \dots, n,
\end{eqnarray}
of the path fibration, i.e. such that the composition 
\begin{eqnarray}
A_i\times A_j \stackrel{\sigma_{ij}}\to F({\Bbb R}^d, n)^I \stackrel\pi\to F({\Bbb R}^d, n)\times F({\Bbb R}^d, n)
\end{eqnarray}
coincides with the inclusion $A_i\times A_j \to F({\Bbb R}^d, n)\times F({\Bbb R}^d, n)$.
Indeed, the desingularization deformation $F_i\times F_j$ takes $A_i\times A_j$ into $A_n\times A_n$; then we apply the deformation (\ref{projection}) which takes 
$A_n\times A_n$ into 
$F(L, n)\times F(L,n)$; and finally we apply section (\ref{sigmal}). 
Let us emphasise that the above description of $\sigma_{ij}$ is totally algorithmic and practically implementable.

%
%
%
%

\subsection{Combining the regions of continuity.} The sets $A_i\times A_j$ where $i, j=1, \dots, n$, are mutually disjoint and cover the whole product $F({\Bbb R}^d, n)\times F({\Bbb R}^d, n)$. 
Over each of these sets we have a continuous section $\sigma_{ij}$; in total we have $n^2$ of these sets. In this subsection we observe that one may combine these sets into $2n-1$ sets 
$W_k$, where $k=2, \dots, 2n$,
such that the sections $\sigma_{ij}$ determine a continuous section over each $W_k$. 

Define
\begin{eqnarray}
W_k \, = \, \bigcup_{i+j=k} A_i\times A_j, \quad \mbox{where} \quad k=2, 3, \dots, 2n.
\end{eqnarray}
We know that the closure of each set $A_i$ is contained in the union of the sets $A_r$ with $r\le i$. 
This implies that for any two distinct pairs $(i, j)$ and $(i', j')$ with $i+j=k=i'+j'$ one has 
$$\overline{A_i\times A_j}\, \cap\,  (A_{i'}\times A_{j'}) \, =\, \emptyset.$$
Therefore  no limit point of $A_i\times A_j$ lies in $A_{i'}\times A_{j'}$ for $i+j=i'+j'$. Hence the sections $\sigma_{ij}$, see (\ref{sigmaij1}), jointly define a continuous section of the path fibration 
$\pi: F({\Bbb R}^d, n)^I \to F({\Bbb R}^d, n)\times F({\Bbb R}^d, n)$
over each set $W_k$. Thus, we have constructed a tame motion planning algorithm in $F({\Bbb R}^d, n)$ having $2n-1$ domains of continuity $W_2, W_3, \dots, W_{2n}$.

\section{A motion planning algorithm in $F({\Bbb R}^d, n)$ with $d\ge 2$ even}\label{sec9}

In this section we improve the motion planning algorithm in $F({\Bbb R}^d, n)$ of the previous section under the assumption that $d\ge 2$ is even. 
This motion planning algorithm will have $2n-2$ domains of continuity. 

For a configuration $C=(z_1, \dots, z_n)\in F({\Bbb R}^d, n)$ consider the line $L'=L'_C$ through the origin which is parallel to the affine line $L=L_C$ connecting the points $z_1$ and $z_2$. 
The line $L_C$ has a natural orientation 
from $z_1$ to $z_2$ and we denote by $e=e_C\in L'_C$ the unit vector $$e_C= \frac{z_2-z_1}{|z_2-z_1|}.$$ Let $p_C: {\Bbb R}^d\to L_C$ denote the orthogonal projection. For a configuration $C=(z_1, \dots, z_n)\in F({\Bbb R}^d, n)$ we denote by $\card(C)$ the cardinality of the set 
$\{p_C(z_1), \dots, p_C(z_n)\}$ of the projection points; note that $\card(C)\in \{2, \dots, n\}$. 

\subsection{Desingularization} For a configuration 
$C\in F({\Bbb R}^d, n)$,  $C=(z_1, \dots, z_n)$
with $\card(C)=i$,
 where $i\ge 2$, denote
$$\epsilon(C)= \frac{1}{n} \min \{|p_C(z_r) -p_C(z_s)|; p_C(z_r)\not= p_C(z_s)\}.$$ 
For $t\in [0,1]$ and $C$ as above define
$F_i(t)(C)=(z_1(t), \dots, z_n(t))$, where $$z_j(t)=z_j+t(j-1)\epsilon(C)e_C$$
for $j=1, \dots, n.$
This gives a \, \lq\lq desingularization\rq\rq\,  deformation $F_i(t)(C)$ with $F_i(0)(C)=C$ and 
$$\card(F_i(t)(C))=n\quad\mbox{for} \quad t\in (0, 1].$$ 
Note that the lines $L_C$ and $L'_C$ do not change under the desingularization, i.e. $L_{F_i(t)(C)}=L_C$ and $L'_{F_i(t)(C)}=L'_C$. 
Besides, the desingularization $F_i(t)(C)$ is continuous as a function of $(t, C)$ if we restrict it to the set of configurations $C$ with $\card(C)=i$ where $i$ is fixed. 

\subsection{Colinear configurations}
For $i, j= 2, \dots, n$ we denote by $A_{ij}$ the set of all pairs of configurations $(C, C')$ where $C, C'\in F({\Bbb R}^d, n)$ such that 
$e_C\not= -e_{C'}$, $\card(C)=i$ and $\card(C')=j$. 
Similarly, for $i, j= 2, \dots, n$ we denote by $B_{ij}$ the set of all pairs of configurations $(C, C')$ where $C, C'\in F({\Bbb R}^d, n)$ such that 
$e_C= -e_{C'}$, $\card(C)=i$ and $\card(C')=j$. 

Clearly, 
\begin{eqnarray}\label{closure}
\overline{B_{ij} }\subset \bigcup_{\stackrel{r\le i} {s\le j}} B_{rs}, \quad \quad \overline{A_{ij} }\subset \bigcup_{\stackrel{r\le i} {s\le j}} A_{rs}\cup 
\bigcup_{\stackrel{r\le i} {s\le j}}B_{rs}.
\end{eqnarray}

Denote by $X\subset F({\Bbb R}^d, n)\times F({\Bbb R}^d, n)$ the set of all pairs $(C, C')$ of configurations such that (a) the vectors $e_C$ and $e_{C'}$ are not opposite to each other, i.e. $e_C \not= -e_{C'}$, and (b) the configurations $C$ and $C'$ are {\it colinear}, i.e. $C\in F(L_C, n)$ and $C'\in F(L_{C'}, n)$.

Consider also the subset $X'\subset X$ consisting of pairs of colinear configurations $(C, C')$ with $e_{C}=e_{C'}$ and $L_C=L_{C'}$. 
%

Besides, we shall denote by $Y\subset F({\Bbb R}^d, n)\times F({\Bbb R}^d, n)$ the set of all pairs of colinear configurations $(C, C')$ such that 
the vectors $e_C$ and $e_{C'}$ are opposite to each other, i.e. $e_C = -e_{C'}$. Note that in this case $L_C=L_{C'}$. 

The union $X\cup Y$ is the set of all pairs of colinear configurations.

\subsection{Deformations $\sigma_{ij}$}\label{sec93}
Next we define the deformations 
\begin{eqnarray}\label{sigmaij}
\sigma_{ij}: A_{ij}\to (F({\Bbb R}^d, n)\times F({\Bbb R}^d, n))^I, \\  \sigma'_{ij}: B_{ij}\to (F({\Bbb R}^d, n)\times F({\Bbb R}^d, n))^I.\nonumber
\end{eqnarray}
deforming $A_{ij}$ into $X$ and $B_{ij}$ into $Y$ correspondingly, i.e.
such that 
\begin{enumerate}
\item $\sigma_{ij}(C, C') (0) = (C, C')$ and $\sigma_{ij}(C, C') (1) \in X$, \\
\item $\sigma'_{ij}(C, C') (0) = (C, C')$ and $\sigma'_{ij}(C, C') (1) \in Y$
\end{enumerate}
Given a pair $(C, C')\in A_{ij}$, we apply first the desingularization deformations $F_i(t)(C)$ and $F_j(t)(C')$  
taking the pair $(C, C')$ to a pair of configurations $(C_1, C'_1)$ with 
$\card(C_1)=n$, $L_{C_1}=L_{C}$ and $\card(C'_1)=n$, $L_{C'_1}=L_{C'}$. 
Next we apply the linear deformation (\ref{projection}) taking the pair $(C_1, C'_1)$ to a pair of colinear configurations 
$(C_2, C'_2)$ where $C_2\in F(L_C, n)$ and $C'_2\in F(L_{C'}, n)$. 
The deformation $\sigma_{ij}$ is the concatenation of the two deformations described above; the deformation $\sigma'_{ij}$ is defined similarly. 

\subsection{}\label{sec94} Next we deform $X$ into $X'$ by a deformation $X\times I\to F({\Bbb R}^d, n)\times F({\Bbb R}^d, n)$ as follows. 
Given two colinear configurations 
$C=(z_1, \dots, z_n)$ and $C'=(z'_1, \dots, z'_n)$ with vectors $e_C$ and $e_{C'}$ satisfying $e_C \not = - e_{C'}$. Making parallel translation, we may assume that both lines $L_C$ and $L_{C'}$ pass through the origin $0\in {\Bbb R}^d$. We may now view $e_C$ and $e_{C'}$ as points of the unit sphere 
$S^{d-1}\subset {\Bbb R}^d$ and, since they are not antipodal, there exists a unique geodesic path $e(t)\in S^{d-1}$ of minimal length connecting them. 
We obtain a continuous path $V_t$ of orthogonal transformations of $V_t: {\Bbb R}^d\to{\Bbb R}^d$, which is identical on the orthogonal complement to the subspace spanned by 
the vectors $e_C$ and $e_{C'}$, and such that $V_t(e_C)=e(t)$. Applying $V_t$ to the configuration 
$C=(z_1, \dots, z_n)$ we get a path $(V_t(z_1), \dots, V_t(z_n))$
in $F({\Bbb R}^d, n)$ taking 
$C$ to a colinear configuration $C"$ such that $e_{C"}= e_{C'}$. 

\subsection{} \label{sec95}
Finally we observe that there exist continuous sections 
\begin{eqnarray}\label{sigmas}
\sigma_{X'}: X'\to F({\Bbb R}^d, n)^I \quad \mbox{and}\quad \sigma_Y: Y\to F({\Bbb R}^d, n)^I
\end{eqnarray}
of the path space fibration 
\begin{eqnarray}\label{fib}
\pi: F({\Bbb R}^d, n)^I \to F({\Bbb R}^d, n)\times F({\Bbb R}^d, n)
\end{eqnarray}
over the sets $X'$ and $Y$ correspondingly. Here we will use our assumption that $d\ge 2$ is even. Let us start with $\sigma_{X'}$. 
Given two colinear configurations 
$C=(z_1, \dots, z_n)$ and $C'=(z'_1, \dots, z'_n)$ with $L=L_C=L_{C'}$ and $e_C=e_{C'}$. The points $z_1, \dots, z_n, z'_1, \dots, z'_n$ lie on the oriented  line $L$ and their "ordering" determines two permutations $(i_1, i_2, \dots, i_n)$ and $(j_1, j_2, \dots, j_n)$ such that $z_{i_1}<z_{i_2}<\dots <z_{i_n}$ and 
$z'_{j_1}<z'_{j_2}<\dots <z'_{j_n}$. Since $d\ge 2$ is even, the unit sphere $S^{d-1}$ admits a continuous and nowhere zero tangent vector field. This means that we may continuously choose a unit vector $e'_C\in S^{d-1}$ perpendicular to $e_C$ for any colinear configuration $C$. 
Now we define the following path in $F({\Bbb R}^d, n)$ which takes $C$ onto $C'$ and is continuous as a function of $(C, t)$; we set
$C^t = (z_1^t, z_2^t, \dots, z_n^t)$ where 
$$z_{i_k}^t = \left\{
\begin{array}{lll}
z_{i_k}+3tk e'_C &\mbox{for}& t\in [0,1/3],\\ \\
z_{i_k}+ke'_C+(3t-1)(z_{j_k}-z_{i_k})&\mbox{for}& t\in [1/3,2/3],\\ \\
z_{j_k} +k(3-3t)e'_C&\mbox{for}& t\in [2/3, 1].
\end{array}
\right.
$$
This formula defines a continuous section of (\ref{fib}) over $X'$ which we shall denote by $\sigma_{X'}$. 
The section $\sigma_Y$, see (\ref{sigmas}), is defined by the similar formulae. 

\subsection{} Now we may concatenate (as explained in example (\ref{exdef})) the deformations of subsections (\ref{sec93}), (\ref{sec94}) and the section $\sigma_{X'}$ (see (\ref{sec95})) to obtain a continuous  section 
$$s_{ij}: A_{ij}\to F({\Bbb R}^d, n)^I$$
of the path fibration over each $A_{ij}$ where $i, j= 2, \dots, n$. Similarly, concatenating the deformation $\sigma'_{ij}$ (see (\ref{sigmaij})) and the section 
$\sigma_Y$ (see (\ref{sigmas})) we obtain a continuous section 
$$s'_{ij}: B_{ij}\to F({\Bbb R}^d, n)^I, \quad i,j=2, \dots, n.$$

\subsection{Repackaging the regions of continuity} The sets $A_{ij}, B_{rs}$ are pairwise disjoint and their union is $F({\Bbb R}^d, n)\times F({\Bbb R}^d, n)$; each of these sets is an ENR and on each of these sets we have constructed a continuous section $s_{ij}$ or $s'_{ij}$. 
%
%
%
%
%
%
%
%

We can define the sets $W_k$ (repackaging) as follows
$$W_k = \bigcup_{i+j=k}A_{ij} \, \cup\, \bigcup_{r+s=k+1}B_{rs}$$
where $k=3, \dots, 2n$. From (\ref{closure}) we see that for $i+j=k$ and $r+s=k+1$ no limit point of $A_{ij}$ may be contained in $B_{rs}$. 
The sections $s_{ij}$ and $s'_{rs}$ define a continuous section of the path fibration over each $W_k$, where $ k=3, \dots, 2n$. 
As the result we obtain $2n-2$ regions of continuity $W_3, W_4, \dots, W_{2n}$; note that $W_3=B_{22}$.

%
%
%
%
%
%
%
%
%
%
%
%
%
%
%
%
%

\section{Configuration Spaces of Graphs}\label{sec10}

\subsection{} Let $\Gamma$ be a connected finite graph. The symbol $F(\Gamma, n)$ denotes the
configuration space of $n$ distinct particles on $\Gamma$. In other words,
$F(\Gamma, n)$ is the subset of the Cartesian product
\begin{eqnarray*}
\underbrace{\Gamma\times \Gamma\times \dots \times \Gamma}_{n \, \,
\mbox{\scriptsize times}} =\Gamma^{ n}
\end{eqnarray*}
consisting of configurations $C=(z_1, z_2, \dots, z_n)$ where $z_i\in \Gamma$ and
$z_i\not= z_j$ for $i\not= j$. The topology of $F(\Gamma,n)$ is induced from its embedding into $\Gamma^{n}$.

Configuration spaces of graphs were studied by R. Ghrist, D. Koditschek and A.
Abrams, see \cite{AA}, \cite{AG},  \cite{Ghrist01}, \cite{GKod}.
To illustrate the importance of these configuration spaces for robotics one may
mention the control problems where a number of automated guided vehicles (AGV)
have to move along a network of floor wires \cite{GKod}. The motion of the
vehicles must be safe: it should be organized so that collisions do not
occur. If $n$ is the number of AGV then the natural configuration space of this
problem is the space $F(\Gamma, n)$ where $\Gamma$ is a graph describing the network of floor wires. 
Here we idealise reality
by assuming that the vehicles have size 0 (i.e. they are points). 


The first question to ask is whether the configuration space $F(\Gamma, n)$ is connected.
Clearly $F(\Gamma, n)$ is disconnected if $\Gamma=[0,1]$ is a closed interval (and $n\geq
2$) or if $\Gamma=S^1$ is the circle and $n\geq 3$. These are the only examples of this
kind as the following simple lemma claims:

\begin{lemma}\label{connected}
Let $\Gamma$ be a connected finite graph having at least one essential vertex.
 Then the configuration
space $F(\Gamma,n)$ is connected.
\end{lemma}

An {\itshape essential vertex}\index{essential vertex} is a vertex of the graph
which is incident to at least 3 edges. We denote the number of essential vertexes of $\Gamma$ by $m(\Gamma)$. 

\subsection{Motion Planning Algorithm in $F(\Gamma,n)$}\label{algorithm} The algorithm presented here was first described in \cite{F3}. 
We assume below that $\Gamma$ is a tree having an essential vertex.
%
Fix a univalent vertex $u_0\in \Gamma$ which will be
called {\it the root}. Any point in $\Gamma$ can be connected by a simple path to the root
$u_0$ and this connecting path is unique up to homotopy. The choice of the root
determines a partial order on $\Gamma$: we say that $x\succeq y$, where $x, y\in \Gamma$
if any path from $x$ to the root $u_0$ passes through $y$. Of course, $\succeq$ is only a
partial order, i.e. there may exist pairs $x, y\in \Gamma$ such that neither $x\succeq
y$, nor $y\succeq x$. On the following picture we see $u\succeq v$ and $w\succeq v$
however $u$ and $w$ are not comparable.
\begin{figure}[h]
\begin{center}
\resizebox{4cm}{5cm}{\includegraphics
{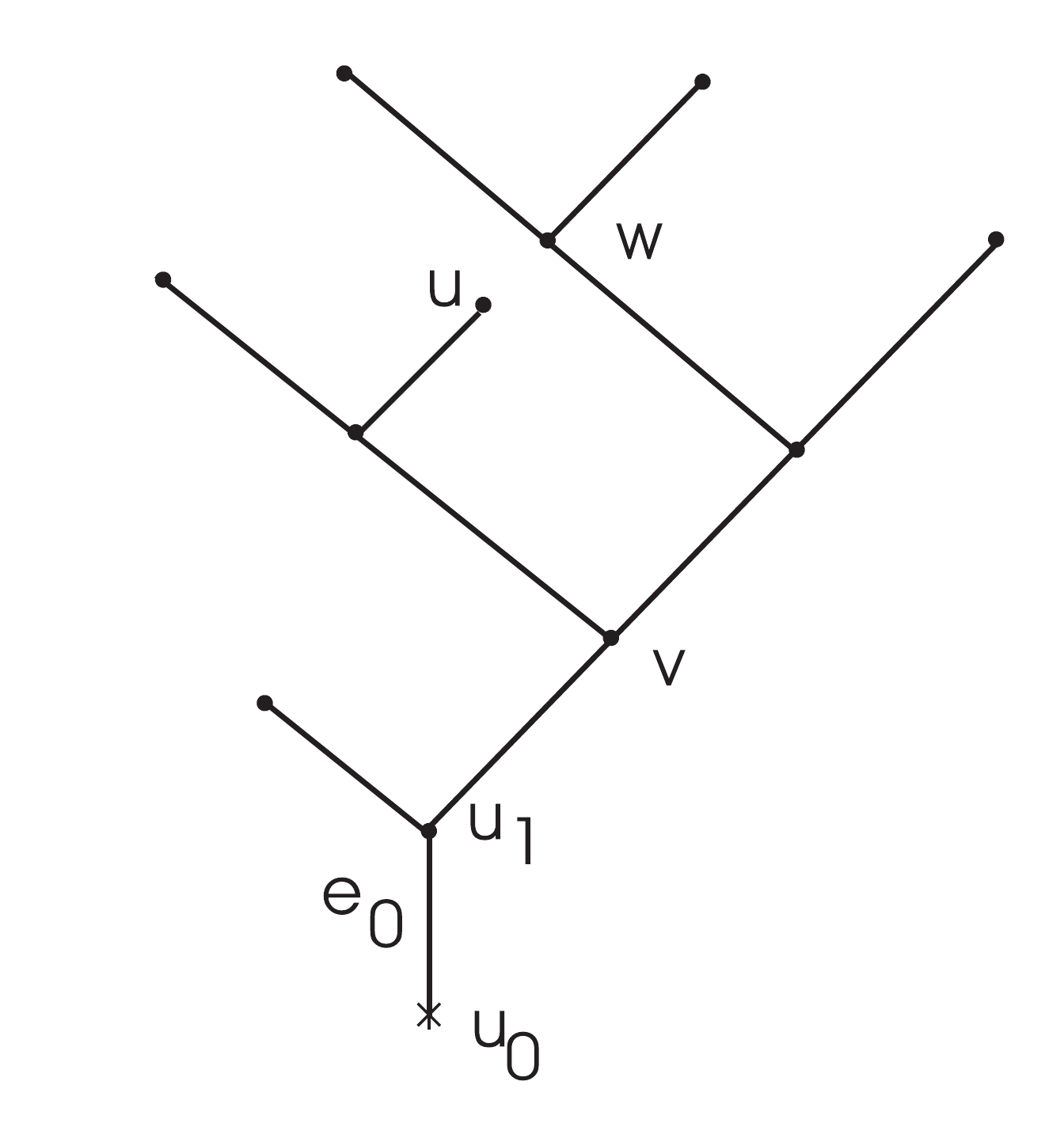}}
\end{center}
\caption{A partial order between the points of the tree.}\label{tree}
\end{figure}

Let $e_0\subset T$ denote the root edge of $\Gamma$. 
For a configuration $C=(z_1, \dots, z_n)\in F(e_0, n)\subset F(\Gamma, n)$ such that $z_i\in e_0$ for all $i=1, \dots, n$ one has 
\begin{eqnarray}\label{order1}
z_{i_1}\preceq z_{i_2} \preceq \dots\preceq z_{i_n}\end{eqnarray}
for some permutation $\tau =(i_1, i_2, \dots, i_n)\in \Sigma_n$. The space $F(e_0, n)$ consists of $n!$ connected components 
labeled by permutations $\tau\in \Sigma_n$, 
$$F(e_0, n) \, =\  \large\bigsqcup_{\tau\in \Sigma_n} F(e_0, n,\tau)$$
 where $F(e_0, n, \tau)$ is the set  of all configurations 
$C=(z_1, \dots, z_n)\in F(e_0, n)$ with the oder of the points $z_j$ described by the permutation $\tau$ as in (\ref{order1}). The fact that each 
space $F(e_0, n,\tau)$ is contractible follows similarly to the arguments of section \ref{secfl}. Using the connectivity of $F(\Gamma, n)$ and the contractibility of 
$F(e_0, n, \tau)$ we may construct a continuous section 
\begin{eqnarray}\label{sigma0}
\sigma_0: F(e_0, n)\times F(e_0, n)\to F(\Gamma, n)^I
\end{eqnarray}
of the path fibration 
$$\pi: F(\Gamma, n)^I \, \to \, F(\Gamma, n)\times F(\Gamma, n)$$
which is similar to (\ref{sigmal}). The section $\sigma_0$ is a continuous motion planning algorithm moving any configuration of $n$ points lying on the root edge $e_0$ to any other such configuration avoiding collisions. Note that under this motion some points will have to leave the root edge before returning to it.


\subsection{} Our algorithm works as follows. Let 
$$A=(A_1, \dots, A_n)\in F(\Gamma, n) \quad \mbox{and}\quad
B=(B_1, \dots, B_n)\in F(\Gamma, n)$$ be two given configurations of $n$ distinct
points on $\Gamma$. Let $A_{i_1}, \dots, A_{i_r}$ be all the minimal elements (with
respect to the order $\succeq$) of the set of points of $A$. Here we assume
that the indices satisfy $i_1<i_2< \dots<i_r$. First we move the point $A_{i_1}$ down to an interior point of
the root edge $e_0$. Next we move $A_{i_2}$ to the root edge $e_0$ and we continue moving
similarly the remaining points $A_{i_3}, \dots, A_{i_r}$ in order of their indices. 
As the result, after this first stage of the algorithm, all the minimal points of $A$ are transferred into the root edge $e_0$. 
On the second stage we find the minimal set
among the remaining points of $A$ and move them down, one after another, to the
edge $e_0$. Iterating this procedure we find a continuous collision free motion of all the
points of $A$ moving them onto the interior of the root edge $e_0$. We obtain a configuration of points $A'=(A_1', \dots, A'_n)\in F(e_0, n)$ 
which all lie
in the interior of the root edge $e_0$, in a certain order.



Applying a similar procedure to the configuration $B$ we obtain a configuration $B'=(B'_1, \dots,
B'_n)\in F(e_0, n)\subset F(\Gamma, n)$ connected with $B$ by a continuous collision free motion. 

Next we apply the section $\sigma_0$ giving a continuous collision free motion from 
$A'$ to $B'$.

Finally, the output of the algorithm is the concatenation of (1) the motion
from $A$ to $A'$; (2) the motion from $A'$ to $B'$ fiven by $\sigma_0$; (3) 
the reverse 
motion from $B$ to $B'$.

\subsection{} The above algorithm has discontinuities: if one of
the points $A_j$ is a vertex $v\in T$ then a small perturbation of $A_j$ inside
$\Gamma$ may lead to a different set of minimal points (see Figure \ref{triple}) and
hence to a completely different ultimate motion. Note that the
vertices of $\Gamma$ which have valence one or two do not cause
discontinuity, i.e. we only need to worry about the essential vertexes of $\Gamma$. 
\begin{figure}[h]
\begin{center}
\resizebox{2cm}{2.5cm}{\includegraphics
{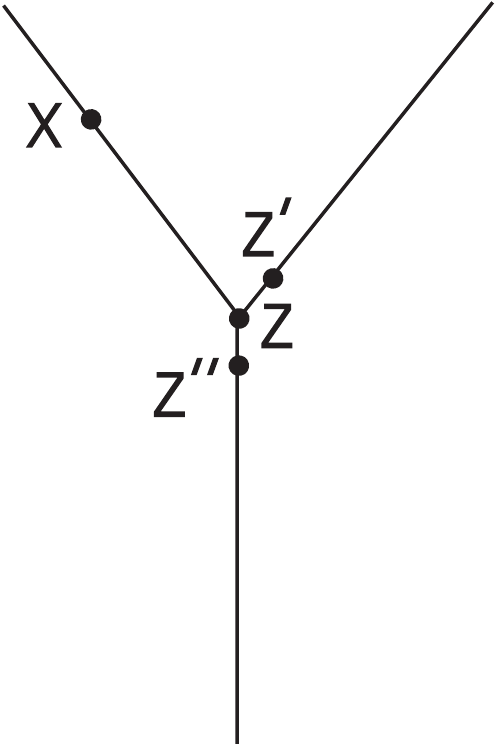}}
\end{center}
\caption{Perturbations $z'$ and $z^{\prime\prime}$ of the vertex point $z$ lead
to different sets of minimal points of the set $\{x, z\}$} \label{triple}
\end{figure}


\subsection{} Let $S_i\subset F(\Gamma, n)$ denote the set of all configurations $A=(A_1,
\dots, A_n)$ such that precisely $i$ points among the points $A_j$ are essential vertices
of $\Gamma$. If we restrict the above algorithm to the set of pairs $(A,B)\in
S_i\times S_j$ with fixed $i, j$, then the result of the algorithm is a continuous function of the input; in other words we have a continuous section 
\begin{eqnarray}
\sigma_{ij}: S_i\times S_j \to F(\Gamma, n)^I, \quad\mbox{where}\quad i, j= 0, 1, \dots, m(\Gamma).
\end{eqnarray}
Recall that $m(\Gamma)$ denotes the number of essential vertices of $\Gamma$. 

\subsection{} We observe that the closure of $S_i$ satisfies
\begin{eqnarray}
\overline S_i \subset \, \bigcup_{j\ge i}S_j.
\end{eqnarray}
It follows that for two distinct pairs $(i,j)$ and $(i', j')$ with $i+j=i'+j'$ one has 
$$\overline{S_i\times S_j}\cap (S_{i'}\times S_{j'}) =\emptyset.$$ 
Hence we obtain that the continuous sections $\sigma_{ij}$ constructed above define a continuous section of the path fibration 
over each set 
$$W_k = \bigcup_{i+j=k} S_i\times S_j, \quad k = 0, 1, \dots, 2m(\Gamma).$$
The sets $W_0, W_1, \dots, W_{2m(\Gamma)}$ form a partition of $F(\Gamma, n)\times F(\Gamma, n)$ and each of these sets is an ENR. Hence we have described a tame motion planning algorithm on $F(\Gamma, n)$ with $2m(\Gamma)+1$ regions of continuity. 

%
%
%
\begin{corollary}\label{cor2}
Let $\Gamma$ be a tree having an essential vertex. Then the topological
complexity of the configuration space $F(\Gamma, n)$ satisfies
\begin{eqnarray}\label{leq}
\tc(F(\Gamma, n)) \leq 2m(\Gamma)+1.
\end{eqnarray}
\end{corollary}

Our goal in the following sections will be to prove the following result:
\begin{theorem}\label{tctree}
Let $\Gamma$ be a tree not homeomorphic to the interval $[0,1]$ and let $n$ be an integer satisfying $n\ge 2m(\Gamma)$; in the case when 
$n=2$ we shall additionally assume that $\Gamma$ is not homeomorphic to the letter Y. Then
\begin{eqnarray}\label{equals}
\tc(F(\Gamma, n)) = 2m(\Gamma)+1.
\end{eqnarray}
\end{theorem}
In other words the upper bound of Corollary \ref{cor2} is exact assuming that $n\ge 2m(\Gamma)$ and
hence the motion planning algorithm described above in this section is optimal. 
There is however one exception: 
if $\Gamma$ is homeomorphic to the letter $Y$ then 
$F(\Gamma,
2)$ is homotopy equivalent to the circle $S^1$ as follows from Theorem \ref{thm1} below. Hence in this case
$\tc(F(\Gamma,2))=2$, see Example \ref{ex22}; the inequality (\ref{leq}) is strict in this
case.

Theorem \ref{tctree} was stated in \cite{F3} without proof. A similar (but slightly different) theorem appears also in a recent preprint \cite{SS}. 

%

\section{The space $F(\Gamma, 2)$ for a tree $\Gamma$}\label{secfg2}

In this section (which can be read independently of the rest of the paper) we describe the $\Z_2$-equivariant homotopy type of the configuration space $F(\Gamma, 2)$ of two distinct particles of a tree $\Gamma$. The involution 
$$\tau: F(\Gamma, 2) \to F(\Gamma, 2)$$ acts by permutting the particles, i.e. $\tau(x, y) =(y, x)$ where $(x, y)\in F(\Gamma, 2)$. 

Recall that {\it the degree} of a vertex $v$  (denoted by $\eta(v)$) is the number of edges of $\Gamma$ incident to $v$. A vertex $v$ is {\it essential} if $\eta(v)\ge 3$. 
Fix a
univalent {\it root vertex} $u_0\in \Gamma$, $\eta(u_0)=1$. Then any vertex $v\not=u_0$ has a
unique {\itshape descending} edge $e$ incident to it; the minimal path connecting $v$ to the root vertex passes through $e$. 
The other $\eta(v) -1$ edges incident to $v$ will be called
{\itshape ascending}.
\begin{figure}[h]
\begin{center}
\resizebox{10cm}{4.5cm}{\includegraphics{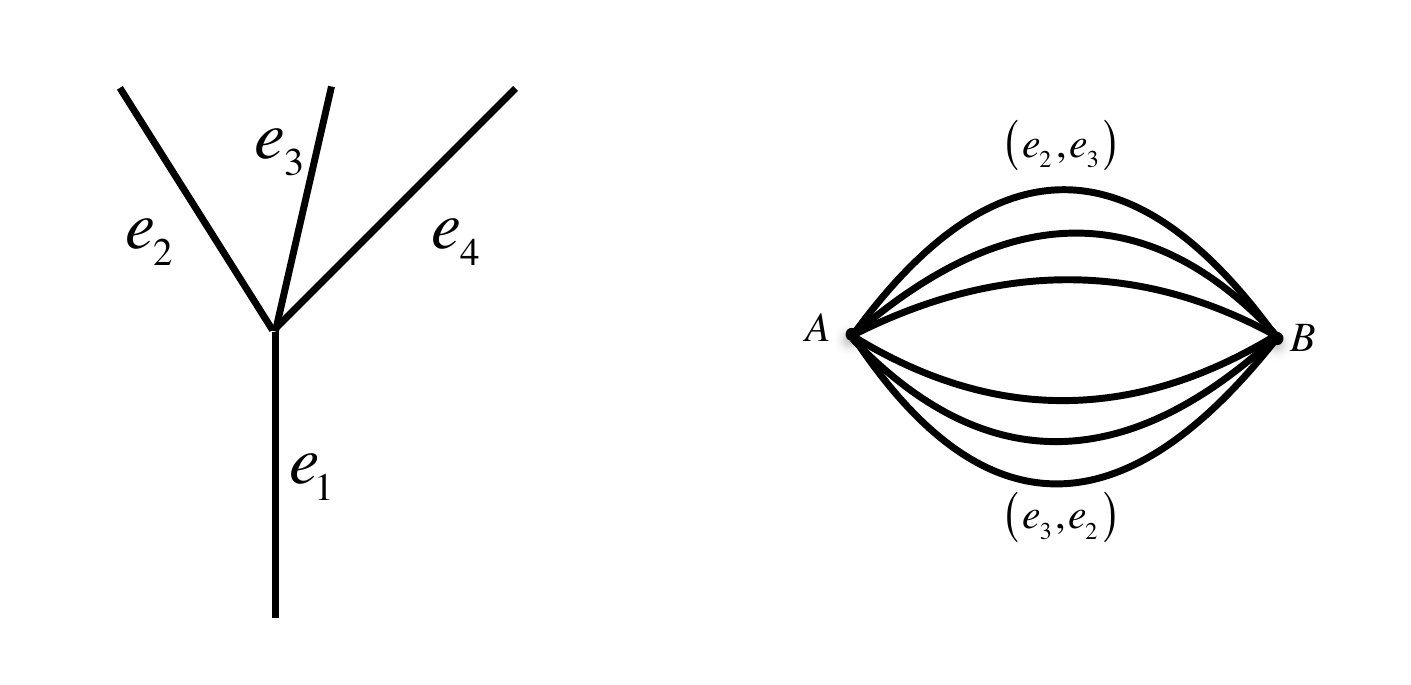}}
\end{center}
\caption{An essential vertex $v$, the descending edge $e_1$ and the ascending edges $e_2, e_3, e_4$ (left). The graph $Q_\Gamma$ (right).} \label{qgamma}
\end{figure}
We associate with a tree $\Gamma$ a
1-dimensional cell complex $Q_\Gamma$ which is constructed as follows. 
The complex $Q_\Gamma$ has two 
$0$-cells
(vertices) $A$ and $B$ and an even number 
$$\sum_{v}(\eta(v)-1)(\eta(v)-2)$$
of
1-dimensional cells connecting $A$ to $B$, each labelled by a triple $(v, e, e')$ where $v$ is an essential vertex of $\Gamma$ and $e, e'$ is an ordered pair of distinct
ascending edges of $\Gamma$ incident to $v$. 
The complex $Q_\Gamma$ has a free 
involution $T:Q_\Gamma\to Q_\Gamma$ which maps $A$ to $B$ and maps homeomorphically each edge with
the label $(v, e, e')$ onto the edge with the label $(v, e',e)$. See Figure \ref{qgamma}.

\begin{theorem}\label{thm1} For a tree $\Gamma$ having an essential vertex
the configuration space $F(\Gamma, 2)$ is ${\mathbf Z}_2$-equivariantly homotopy equivalent
to the complex $Q_\Gamma$.
\end{theorem}

This theorem was stated in \cite{F3} without proof. The configuration spaces $F(\Gamma, 2)$ for various classes of graphs $\Gamma$ complementing the class of trees were explicitly described in \cite{BF}, \cite{FH}. 


\begin{proof}[Proof of Theorem \ref{thm1}]
We repeat the standard arguments (compare Proposition 4G.2 from \cite{Hat02}) emphasising the equivariant features we are dealing with. 

First we describe an open cover 
$F(\Gamma, 2) \, =\, U\cup V.$
Since $\Gamma$ is a tree, for any two points $x, x'\in \Gamma$ there exists a unique simple path in $\Gamma$ connecting $x$ to $x'$. Fix an interior point $u'_0$ of the edge incident to the root vertex $u_0$.
Denote by 
$U$ the set of all configurations $(x, y)\in F(\Gamma, 2)$ such that the simple path connecting $x$ to $u_0$ does not pass through $y$. 
Similarly, we denote by $V\subset 
F(\Gamma, 2)$ the set of all configurations $(x, y)$ such that the simple path connecting $y$ to the root $u_0$ does not pass through $x$. It is obvious that $U$ and $V$ are open and cover $F(\Gamma, 2)$. 

The set $U$ is contractible. Indeed, if $(x, y)\in U$ then we may move the configuration $(x, y)$ continuously to the configuration $(u_0, u'_0)$ by first moving $x$ along the minimal path to $u_0$ and then moving $y$  along the minimal path to $u'_0$. We obtain a path $(x(t), y(t))\in F(\Gamma, 2)$ (where $t\in [0,1]$) 
with $(x(0), y(0))=(x, y)$ and $(x(1), y(1))=(u_0, u'_0)$ which is not only continuous as function of $t$, but it is also  depends continuously on the initial pair 
$(x, y)$. Therefore we obtain a continuous deformation retraction of the set $U$ to the point $(u_0, u'_0)$. 

Similarly, the set $V$ is contractible. 

A configuration $(x, y)\in F(\Gamma, 2)$ lies in the intersection $U\cap V$ if the minimal path connecting $x$ to $u_0$ does not pass through $y$ and the minimal path connecting $y$ to $u_0$ does not pass through $x$. Initially, these two minimal paths have distinct routes before they meet at an essential vertex $v$ and then they coincide and follow the minimal path connecting $v$ 
to $u_0$ (see Figure \ref{wee}). 
\begin{figure}[h]
\begin{center}
\resizebox{6cm}{6cm}{\includegraphics{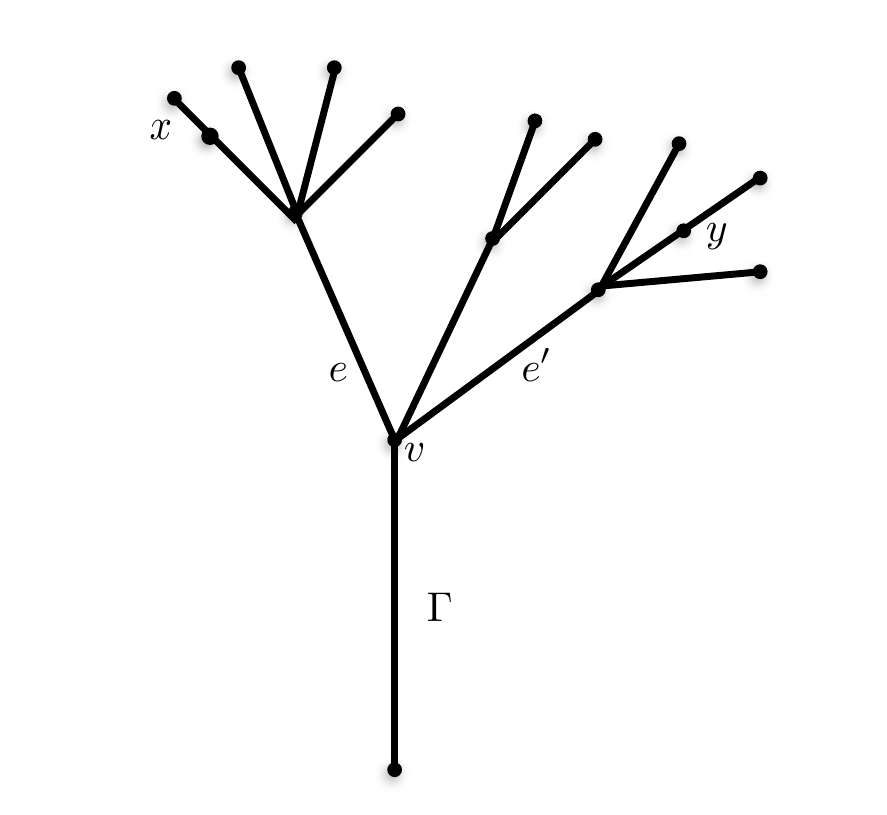}}
\end{center}
\caption{A configuration $(x,y)$ lying in the set $W_{v, e, e'}$.} \label{wee}
\end{figure}
We see that the intersection $U\cap V$ has many connected components which can be labelled by triples $(v, e, e')$ where 
$v$ is an essential vertex of $\Gamma$ and $e, e'$ is an ordered pair of ascending edges incident to $v$. We denote by $W_{v, e,e'}$ the set of configurations $(x, y)\in F(\Gamma, 2)$ such that the minimal path connecting $x$ to the root $u_0$ contains an internal point of $e$ and the minimal path connecting $y$ to the root $u_0$ contains an internal point of $e'$. The sets $W_{v,e, e'}$ corresponding to different triples $v, e, e'$ are disjoint and 
$$U\cap V \, =\, \bigsqcup_{v,e,e'} \, W_{v,e, e'}.$$
Each set $W_{v,e, e'}$ is contractible since one may continuously move any configuration $(x, y)\in W_{v,e, e'}$ into 
a fixed configuration $(x_0, y_0)$, where $x_0\in e$ and $y_0\in e'$, by moving $x$ and $y$ along the minimal paths 
connecting these points to the root $u_0$; this motion is continuous both as a function of time $t$ and as a function of the initial conditions $(x, y)$. 

The involution $\tau: F(\Gamma, 2)\to F(\Gamma, 2)$, where $\tau(x, y)=(y, x)$, maps $U$ onto $V$ and vice versa. Besides, $\tau$ maps each set $W_{v, e, e'}$ homeomorphically onto the set $W_{v, e', e}$. 

Consider the quotient $Q$ of the disjoint union 
$$U\sqcup V \sqcup \left((U\cap V)\times [0,1]\right)$$
where for each configuration $(x, y)\in U\cap V$ we identify the point $((x, y), 0)$ with $(x, y)\in U$ and the point $((x, y), 1)$ with $(x, y)\in V$. The quotient $Q$ carries a natural involution 
$$\tau[(x, y), t]= [(y, x), 1-t],$$
where the square brackets denote the equivalence class in $Q$.
The projection map 
$P: Q\to F(\Gamma, 2)$ is $\Z_2$-equivariant. Here for $(x, y)\in U\cap V$ one has $P[(x, y),t]=(x, y)$; similarly, for 
$(x, y)\in U$ or $(x, y)\in V$ one has $P(x, y)=(x, y)$. 

Next we show that there exists a continuous $\Z_2$-equivariant section $S: F(\Gamma, 2)\to Q$. 
Indeed, fix a partition of unity subordinate to the cover 
$U, V$; it is a pair of continuous functions $\phi_U, \phi_V: F(\Gamma, 2) \to [0,1]$ such that 
${\supp}\, (\phi_U)\subset U$, ${\supp}\, (\phi_V)\subset V$ and for each configuration $(x, y)\in F(\Gamma, 2)$ one has 
$$\phi_U(x, y) +\phi_V(x, y) =1.$$
Define the functions
$\psi_U, \psi_V: F(\Gamma, 2) \to [0,1]$ by
$$\psi_U(x, y) = \frac{1}{2}\cdot\left(\phi_U(x, y) + \phi_V(y,x)\right),$$
$$\psi_V(x, y) = \frac{1}{2}\cdot\left(\phi_V(x, y) + \phi_U(y,x)\right).$$
This is a partition of unity satisfying the additional property $\psi_U(x, y) = \psi_V(y, x)$. 
We may define the section $S: F(\Gamma, 2)\to Q$ by setting 
$$S(x, y) \, =\, \left[(x,y), \psi_U(x, y)\right].$$ 
We have 
$S(y,x)=\left[(y,x), 1-\psi_U(x, y)\right],$ i.e. $S$ is $\Z_2$-equivariant. Clearly, $P\circ S$ is the identity map. Besides, the homotopy 
$\Phi_\tau: Q\to Q$ given by 
$$\Phi_\tau \left[(x, y), t\right] = 
\left[(x, y), (1-\tau)\psi_U(x, y) + \tau t\right]$$
connects the identity map and the composition $S\circ P$. This shows that $P$ is an equivariant deformation retraction. 

Let $Q'$ denote the quotient of $Q$ where the set $U$ is collapsed to a single point (denoted $A$) and the set $V$ is collapsed to another single point denoted $B$. The space $Q'$ is the unreduced suspension 
$$\Sigma(U\cap V)\, = \, \Sigma(\large\bigsqcup_{(v, e, e')} W_{v,e, e'}).$$
Since $U$ and $V$ are contractible, we see that the quotient map $Q\to Q'$ is an equivariant homotopy equivalence. 

Next we use the fact that each set $W_{v, e, e'}$ is contractible, see above. Hence we obtain that $Q'$ equivariantly deformation retracts onto the suspension $\Sigma(\{(v, e, e')\})$
where $\{(v, e, e')\}$ is viewed as a discrete set of labels. Clearly, the suspension $\Sigma(\{(v, e, e')\})$ coincides with 
what we denoted by $Q_\Gamma$. 
Thus we have constructed a chain of equivariant homotopy equivalences
$F(\Gamma, 2)\simeq Q\simeq Q'\simeq Q_\Gamma$. 
This completes the proof. 
\end{proof}

\begin{ex}\label{ex1} {\rm Consider the graph $\Gamma $ of the letter Y which has a single essential vertex of degree $3$. Applying Theorem 
\ref{thm1} we obtain that the configuration space $F(\Gamma, 2)$ is equivariantly homotopy equivalent to the circle $S^1$ with the standard antipodal involution.}
\end{ex}

The following result is a straightforward corollary of Theorem \ref{thm1}. 

\begin{corollary}\label{thm2}
Let $\Gamma, \Gamma'$ be trees such that $m(\Gamma')>0$ and let $\alpha: \Gamma' \to \Gamma$ be a topological embedding. Then the natural inclusion 
$\alpha': F(\Gamma, 2)\to F(\Gamma, 2)$ induces a monomorphism 
$$\alpha'_\ast: H_1(F(\Gamma', 2))\to H_1(F(\Gamma, 2)).$$
\end{corollary}

\begin{proof} Let $v_0$ be a root vertex of $\Gamma'$, and let $u_0$ be a root vertex of $\Gamma$ such that the path connecting $u_0$ to $\alpha(v_0)$ is disjoint from  
$\alpha(\Gamma'-\{v_0\})$. 
We see that the complex $Q_{\Gamma'}$ is naturally a subcomplex of $Q_\Gamma$ which implies our statement due to Theorem \ref{thm1}.
\end{proof}

\begin{corollary} 
If $\Gamma'$ is the graph homeomorphic to the letter $Y$ then $H_1(F(\Gamma', 2))=\Z$ and each topological embedding 
$\alpha: \Gamma'\to \Gamma$ determines a generator of the group $\Z\simeq \alpha_\ast(H_1(F(\Gamma', 2))\subset H_1(F(\Gamma, 2))$, unique up to a sign. 
The homology classes corresponding to all such embeddings $\alpha$ generate the group $H_1(F(\Gamma, 2))$ (not freely). 
\end{corollary}

\section{Top-dimensional cohomology of $F(\Gamma, n)$}\label{sec12}

In this section we utilise the results of \S \ref{secfg2} to construct useful cohomology classes of $F(\Gamma, n)$ of the top dimension. The results of this section will be used in the proof of Theorem \ref{tctree}.

Let $\Gamma$ be a tree with $m=m(\Gamma)$ essential vertices. It is known for any $n$ that the configuration space $F(\Gamma, n)$ has the 
homotopy type of a cell complex of dimension $\le m$; in particular $H^i(F(\Gamma, n))=0$ for $i>m$, see \cite{Ghrist01}. In this section we shall consider the $m$-dimensional cohomology classes 
of $F(\Gamma, n)$ assuming that $n\ge 2m$.

\subsection{} We start from the following general remark which will be useful in the sequel. 
\begin{lemma}\label{dom} Let $\Gamma$ be a connected graph having a univalent vertex. 
Then for any $n'>n$ 
the natural projection $$p: F(\Gamma, n') \to F(\Gamma, n),$$ where 
$(x_1, \dots, x_{n'})\mapsto (x_1, \dots, x_n)$,
is a domination, i.e. there exists a continuous map 
$q: F(\Gamma, n) \to F(\Gamma, n')$ such that $p\circ q \sim \id$.
In particular $p$ induces a monomorphism 
$p^\ast: H^\ast(F(\Gamma, n))\to H^\ast(F(\Gamma, n')).$
\end{lemma} 
\begin{proof} 
Let $u_0$ be a univalent vertex of $\Gamma$. Let $U\subset \Gamma$ be a small open neighbourhood of $u_0$ in $\Gamma$. 
The graph $\Gamma'=\Gamma -U$ is homeomorphic to $\Gamma$. Define the map 
$s: F(\Gamma',n)\to F(\Gamma, n')$
as follows: fix a set of $n'-n$ pairwise distinct points $a_{n+1}, a_{n+2}, \dots, a_{n'}\in U$ and for any configuration $C=(z_1, \dots, z_n)\in F(\Gamma',n)$ define
$s(C)=(z'_1, \dots, z'_{n'})\in F(\Gamma, n')$ by $z'_i=z_i$ for $i\le n$ and $z'_i=a_i$ for $i>n$. The composition
$$F(\Gamma', n)\stackrel{s}\to F(\Gamma, n')\stackrel p\to F(\Gamma, n)$$
is a homeomorphism induced by the inclusion $\Gamma' \to \Gamma$.
We observe that there is a homotopy of injective maps $h_t: \Gamma\to \Gamma$ with $t\in [0,1]$ such that 
$h_0=1_\Gamma$ and $h_1(\Gamma)=\Gamma'$. Then the composition 
$$F(\Gamma, n) \stackrel{h_1}\to F(\Gamma', n)\stackrel{s}\to F(\Gamma, n')\stackrel p\to F(\Gamma, n)$$
is homotopic to the identity (through the homotopy $h_t: F(\Gamma, n)\to F(\Gamma, n)$). Thus $p\circ q\sim \id$ where $q= s\circ h_1$. 
\end{proof}

\subsection{} From here on, let the symbol $\Gamma$ denote a tree. 
For $n\ge 2m$, let $$\Phi_i: F(\Gamma, n)\to F(\Gamma, 2)\quad\mbox{where}\quad i=1, \dots, m$$ denote the projection
$$\Phi_i(x_1, \dots, x_n) =(x_{2i-1}, x_{2i}).$$

\subsection{} Denote by $v_1, \dots, v_m$ the essential vertices of $\Gamma$. For each $j=1, \dots, m$ fix a topological embedding 
$\Gamma_j\subset \Gamma$ of a letter $Y$ graph into $\Gamma$ around the essential vertex $v_j$. 
Besides, let $\Gamma_0\subset \Gamma$ be a small interval containing the root vertex. 
We assume that the subtrees $\Gamma_0, \dots, \Gamma_m$ are sufficiently small so that 
$\Gamma_i\cap \Gamma_j=\emptyset$ for $i\not=j$. 
We shall consider the space $F(\Gamma_i, 2)$ as being a subspace of $F(\Gamma, 2)$ for each $i$. 

Define the embedding
\begin{eqnarray}
\Psi: \prod_{i=1}^m F(\Gamma_i, 2) \to F(\Gamma, 2m)
\end{eqnarray}
by
$$((x_1, x_2), (x_3, x_4), \dots, (x_{2m-1}, x_{2m})) \mapsto (x_1, x_2, x_3, \dots, x_{2m}).$$
We shall denote by $T^m\subset F(\Gamma, 2m)$ the image of $\Psi$. It is a subset homotopy equivalent to the $m$-dimensional torus, see Example \ref{ex1}. 

We have the commutative diagram
\begin{eqnarray}\label{cd1}
\begin{array}{ccc}
\prod_{j=1}^m F(\Gamma_j, 2) & \stackrel \Psi\to & F(\Gamma, 2m)\\ \\
\downarrow \Pi_i && \downarrow \Phi_i\\ \\
F(\Gamma_i, 2) & \to & F(\Gamma, 2).
\end{array}
\end{eqnarray}
Here $\Pi_i$ is the projection on the $i$-th factor and the lower horizontal map is the inclusion. 

\subsection{}\label{sec24} For any $j=1, \dots, m$ choose a cohomology class
$$\alpha_j\in H^1(F(\Gamma, 2)),$$
which is associated with the vertex $v_j$ via Theorem \ref{thm1}; more specifically, we require that 
\begin{eqnarray}
\begin{array}{l}
\alpha_i| F(\Gamma_i, 2) \not=0\in H^1(F(\Gamma_i, 2))=\Z 
\\ \\
\alpha_j| F(\Gamma_i, 2) =0 \quad \mbox{if}\quad i\not=j.
\end{array}
\end{eqnarray}
Such classes exist due to Theorem \ref{thm1}. 

We obtain $m^2$ cohomology classes 
$$u_{ij}\in H^1(F(\Gamma, 2m)),$$
defined by
$$ u_{ij} = \Phi^\ast_i(\alpha_j), \quad i, j= 1, \dots, m.$$
Using the commutative diagram (\ref{cd1}) we obtain 
\begin{eqnarray}\label{restriction}
\begin{array}{l}
u_{ii}|_{T^m} \not=0\in H^1(T^m) \\ \\
u_{ij}|_{T^m} =0 \quad\mbox{if}\quad i\not=j.
\end{array}
\end{eqnarray}
Moreover, we see that the cup-product 
\begin{eqnarray}\label{prod1}
u_{11}u_{22}\dots u_{mm} =  \prod_{i=1}^m u_{ii} \, \in \, H^{m}(F(\Gamma, 2m))
\end{eqnarray}
is nonzero since $$(\prod_{i=1}^m u_{ii}) \, | _{T^m}\not=0.$$ This follows from our remark above that the class $u_{ii}$ 
is induced from a nonzero class $\alpha_i |F(\Gamma_i, 2)$ under the projection $\Pi_i: T^m\to F(\Gamma_i, 2)$. 

If $z\in H_m(F(\Gamma, 2m))$ denotes the homology class realised by $T^m$ then 
\begin{eqnarray}\label{prod2}
\langle \prod_{i=1}^m u_{ii},\, \  z\rangle \not=0.
\end{eqnarray}

\subsection{}
Next we consider different $m$-fold products of the classes $u_{ij}$. First we observe that 
$ u_{ij}u_{ik}=0$
for any $i, j, k$. 
Indeed, $\alpha_j\alpha_k=0\in H^2(F(\Gamma, 2))$ since $F(\Gamma, 2)$ is homotopy equivalent to a graph; 
hence $u_{ij}u_{ik}=\Phi_i^\ast(\alpha_j\alpha_k)=0.$

%

Let $\sigma=(i_1, i_2, \dots, i_m)$ be a sequence with $i_k\in  \{1, 2, \dots, m\}$; we do not require it to be a permutation, i.e. repetitions of the 
indices are allowed. We associate with $\sigma$ the top-dimensional cohomology class $$u_\sigma= u_{1 i_1}u_{2 i_2}\dots u_{m i_m}\in H^m(F(\Gamma, 2m)).$$
It follows from (\ref{restriction}) that 
\begin{eqnarray}\label{evaluation}
\langle u_\sigma, z\rangle = 0,\end{eqnarray}
assuming that $\sigma$ is distinct from the sequence $(1, 2, \dots, m)$. 

\subsection{} For a permutation $\tau=(j_1, j_2, \dots, j_m)$ of the indices $ 1, 2, \dots, m$ define the homeomorphism 
$$L^\tau : F(\Gamma, 2m) \to F(\Gamma, 2m)$$ by 
$L^\tau (x_1, x_2, \dots, x_{2m}) = (x_{2j_1-1}, x_{2j_1}, \dots, x_{2j_m-1}, x_{2j_m}).$ Define also the homology class
$$z^\tau\, =\, L^\tau_\ast (z) \, \in H_m(F(\Gamma, 2m)).$$

For a sequence $\sigma=(i_1, i_2, \dots, i_m)$ and for a permutation $\tau=(j_1, j_2, \dots, j_m)$, we claim that the evaluation
\begin{eqnarray}\label{prod4}
\langle u_\sigma, \, z^\tau\rangle \, \not= 0
\end{eqnarray}
is nonzero if and only if $\sigma$ and $\tau$ coincide. Indeed, we compute
\begin{eqnarray*}
\langle u_\sigma, z^\tau \rangle  &=& \langle u_\sigma, \, L^\tau_\ast(z)\rangle = \\
&=& \langle (L^\tau)^\ast (u_\sigma), z\rangle  = \langle (L^\tau)^\ast (\prod_{k=1}^m u_{k i_k}), \, z\rangle
\\
&=& \langle  \prod_{k=1}^m (L^\tau)^\ast(u_{k i_k}), \, z\rangle
= \, \langle \prod_{k=1}^m (\Phi_k \circ L^\tau)^\ast(\alpha_{i_k}), \, z\rangle\\
&=& \langle \prod_{k=1}^m \Phi_{j_k}^\ast(\alpha_{i_k}), z\rangle = \langle \prod_{k=1}^m u_{j_k i_k}, z\rangle.
\end{eqnarray*}
Here we used that $\Phi_k\circ L^\tau= \Phi_{j_k}$. 
Using (\ref{restriction}) and (\ref{prod1}) we obtain that the number $\langle u_\sigma, z^\tau \rangle$ is nonzero iff 
$j_k=i_k$ for any $k$, i.e. iff $\sigma$ and $\tau$ are equal. 

\begin{corollary}
The cohomology classes $u_\sigma\in H^m(F(\Gamma, 2m))$ corresponding to various permutations $\sigma$ are linearly independent. In particular, for $n\ge 2m$ the rank of the group $H^m(F(\Gamma, n))$ is at least $m!$. 
\end{corollary}

\section{Proof of Theorem \ref{tctree}} \label{sec13}
Below we assume that $\Gamma$ is a tree and $n\ge 2m$. Let us first assume that $m\ge 2$. 

Any degree one cohomology class $u\in H^1(F(\Gamma,n))$ determines a zero-divisor 
$$\bar u = u\otimes 1-1\otimes u \in H^\ast(F(\Gamma, n))\otimes H^\ast(F(\Gamma, n)).$$
Our goal is to find $2m$ cohomology classes of degree one such that the product of the corresponding zero-divisors is nonzero. 

We shall use the notations introduced in the previous section. Consider the classes $u_{11}, u_{22}, \dots, u_{mm}$ and 
$u_{12}, u_{23}, \dots, u_{(m-1) m}, u_{m1}$ and the corresponding zero-divisors $\bar u_{ii}, \bar u_{i(i+1)}$. We want to show that the product 
\begin{eqnarray}\label{prod3}
\prod_{i=1}^m \bar u_{ii}\times \prod_{i+1}^m \bar u_{i(i+1)} \not\, =\, 0\end{eqnarray}
is nonzero. We know that the cohomology of $F(\Gamma, n)$ 
vanishes in degrees $>m$. Therefore we obtain
$$\prod_{i=1}^m \bar u_{ii}\times \prod_{i+1}^m \bar u_{i(i+1)}  = 
\sum_S \pm \left(\prod_{i\in S} u_{ii}\times \prod_{i\not\in S} u_{i(i+1)}\right)\, \otimes\, 
\left(\prod_{i\not\in S}u_{ii}\times \prod_{i\in S}u_{i(i+1)}\right).$$
Here $S$ runs over all subsets $S\subset \{1, 2, \dots, m\}$; the sign $\times$ denotes the cup-product. 
We shall evaluate the product (\ref{prod3}) on the tensor product of two homology classes $z\otimes z^\tau$ where 
$\tau$ is the permutation $(2, 3, \dots, m,1)$ and $z\in H_m(F(\Gamma, n))$ is the homology class defined towards the end of subsection \ref{sec24}. Using statements (\ref{prod2}), (\ref{evaluation}), (\ref{prod4}) we find that all the terms in the sum
\begin{eqnarray*}
&&\langle \prod_i \bar u_{ii}\otimes \prod_i \bar u_{i(i+1)} ,  z\otimes z^\tau\rangle =\\
&&\sum_S \pm \langle \prod_{i\in S}u_{ii}\times \prod_{i\not\in S} u_{i(i+1)}, z\rangle \cdot \langle 
\prod_{i\not\in S}u_{ii}\times \prod_{i\in S} u_{i(i+1)}, z^\tau\rangle
\end{eqnarray*}
vanish except for the term with  
$S=\{1, 2, \dots, m\}$ which is nonzero. This shows that the product (\ref{prod3}) is nonzero. Therefore, 
$${\zcl} (F(\Gamma, n)\ge 2m \quad \mbox{for}\quad n\ge 2m.$$
By Theorem \ref{thm4} we have $\tc(F(\Gamma, n))\ge 2m+1$ and the inverse inequality is given by Corollary \ref{cor2}.

The above arguments fail in the case $m=1$, i.e. when $\Gamma$ is a tree with a single essential vertex $v$, $\eta(v)\ge 3$. In the case $\eta(v)=3$ the tree $\Gamma$ is homeomorphic to the letter $Y$, hence we shall assume that $\eta(v)\ge 4$. Consider the graph $Q_\Gamma$ given by Theorem \ref{thm1}. By Theorem \ref{thm1} the space $F(\Gamma, 2)$ is homotopy equivalent to a wedge of circles where the number of circles equals 
$$b_1(F(\Gamma, 2))= \sum_{v\in V(\Gamma)} \left(\eta(v)-1\right)\left(\eta(v)-2\right) - 1\ge 5.$$
Using Example \ref{exgraph} we have $\tc(F(\Gamma, 2))=3$. This proves our statement for $n=2$. 
If $n>2$ we apply Lemma \ref{dom} and Theorem \ref{domination} to conclude $\tc(F(\Gamma, n))\ge \tc(F(\Gamma, 2))=3$. 
This completes the proof. 

\section{Further comments}\label{sec14}

\subsection{} It is interesting to compare Theorems \ref{thm11} and \ref{tctree}. The topological complexity 
$\tc(F({\Bbb R}^d, n))$ is linear in $n$ but, in contrast, $\tc(F(\Gamma, n))$ equals $2m(\Gamma)+1$, i.e. it is independent of $n$. 
This result may have some practical implications: {\it to simplify the task of controlling a large number of objects moving in space without collisions 
one may restrict their motion to a graph}.

\subsection{} In \cite{FGY} the authors analysed the topological complexity of collision free motion planning of multiple objects in ${\Bbb R}^d$ in the presence of moving obstacles. 

\subsection{} The notion of {\it higher topological complexity} $\tc_s(X)$, where $s= 2, 3, \dots$ was introduced by Rudyak \cite{Rudyak}. The number 
$\tc_s(X)$ can be defined as the Schwarz genus of the fibration 
$$p_s: X^I\to X^s$$
where 
$$p_s(\gamma)= \left(\gamma(0), \gamma\left(\frac{1}{s-1}\right), \dots, \gamma\left(\frac{k}{s-1}\right), \dots, \gamma(1)\right),$$
compare (\ref{fibration}). The invariant $\tc(X)$ which we studied in this paper 
coincides with $\tc_2(X)$. 

The invariant $\tc_s(X)$ is also related to robotics: while in the case of $\tc(X)$ we are dealing with algorithms for a robot to move from an initial state to a final state, in the case of $\tc_s(X)$ with $s>2$ we require that while moving from the initial state to the final state the robot visits $s-2$ additional intermediate states. 
This explains why $\tc_s(X)$ is also called {\it "the sequential topological complexity". }

Note that our notation $\tc_s(X)$ stands for what is called {\it "the unreduced"} topological complexity;  {\it "the reduced"} version is smaller by one. 

\subsection{} The sequential topological complexities of configuration spaces $F({\Bbb R}^d, n)$ were computed in \cite{GG}: 
$$\tc_s(F({\Bbb R}^d, n))= \left\{
\begin{array}{ll}
sn-s+1, & \mbox{for $d$ odd},\\ \\
sn-s, & \mbox{for $d$ even}.
\end{array}
\right.
$$

\subsection{} The topological complexity of a closed orientable surface $\Sigma_g$ of genus $g$ was computed in the initial paper \cite{F1}: 
$$
\tc(\Sigma_g) = 
\left\{
\begin{array}{ll}
3, & \mbox{for $g=0 $ and $g=1$},\\ \\
5, & \mbox{for $g\ge 2 $}.
\end{array}
\right.
$$
The task of finding $\tc(N_g)$ turned out to be much more difficult;
here $N_g$ stands for the closed non-orientable surface of genus $g$. The case $N_1$
(the real projective plane) was settled in \cite{FTY}: $$\tc(N_1)=4.$$ A. Dranishnikov \cite{Dra1} proved that 
$\tc(N_g)=5$ for any $g\ge 5$; he also mentioned that his method can be pushed to prove that $\tc(N_4) =5$ as well. 
While preparing this paper for publication (December 2016) I received information that two independent groups of researchers obtained the full solution to the problem: 
$$\tc(N_g)=5, \quad \mbox{for any $g \ge 2$},$$
see \cite{CV} and \cite{DD}.

%
%
%
%
%


\bibliographystyle{amsalpha}

\end{document}